\documentclass[11pt]{article}

\usepackage[margin=1in]{geometry}
\usepackage{amsmath,amssymb,amsthm}
\usepackage{bm,bbm}
\usepackage{graphicx}
\usepackage{xcolor}
\graphicspath{{./figures/}}
\usepackage{subcaption}
\usepackage{tikz}
\usepackage{algorithm}
\usepackage{algpseudocode}
\usepackage{hyperref}

\numberwithin{equation}{section}
\numberwithin{figure}{section}
\numberwithin{table}{section}
\numberwithin{algorithm}{section}

\newtheorem{theorem}{Theorem}[section]
\newtheorem{lemma}[theorem]{Lemma}

\theoremstyle{definition}

\theoremstyle{remark}
\newtheorem{remark}[theorem]{Remark}

\title{Solution of a system of Hamilton-Jacobi equations on a network\thanks{Parts of this manuscript were edited with the assistance of an AI tool (Claude Opus 5.5, Anthropic). The authors are solely responsible for all proofs and statements.}}
\author{Emily R. Hammett\thanks{Department of Mathematics, The Pennsylvania State University, University Park, PA, USA (\texttt{erh5508@psu.edu}).}
\and Maria Teresa Chiri\thanks{Department of Mathematics and Statistics, Queen's University, Kingston, ON, Canada (\texttt{maria.chiri@queensu.ca}).}
\and Kenneth D. Czuprynski\thanks{Applied Research Laboratory, The Pennsylvania State University, University Park, PA, USA (\texttt{kdc168@psu.edu}).}
\and Ludmil T. Zikatanov\thanks{Department of Mathematics, The Pennsylvania State University, University Park, PA, USA (\texttt{ludmil@psu.edu}).}}
\date{}

\begin{document}

\maketitle

\begin{abstract}
We extend the study of Hamilton-Jacobi equations on networks from a single equation to a weakly coupled system of two equations on a star graph, and we develop a numerical scheme for it. Instead of prescribing the boundary values and imposing a junction condition, we compute the solution on each branch from its outer vertex toward the junction and determine the boundary values for which these branchwise solutions agree at the junction. On each branch, we replace the right-hand sides by their cell averages and solve the resulting problems exactly on every cell. This gives an iteration whose limits we compute in closed form. From these limits, the values of the discrete solution at the junction are affine functions of the boundary values on the branch, with a $2 \times 2$ matrix whose rows sum to one and whose determinant factors into explicit contributions of the cells. We use these formulas to characterize the boundary values that give solutions continuous at the junction and monotone on every branch: when the matrices are nonsingular, the admissible boundary values of one branch form a convex polygon, bounded by conditions on the difference of the two components and by positivity conditions, and the boundary values on the other branches follow from them. We also treat the singular case, and we illustrate the results with numerical experiments.
\end{abstract}

\section{Introduction}\label{sec:intro}
We consider weakly coupled systems of Hamilton-Jacobi equations on a network. Schieborn and Camilli \cite{SchiebornCamilli} introduced the network problem for one equation; see also \cite{AchdouCamilliCutriTchou2013, SiconolfiSorrentino2018}. The theory of junction conditions for a single equation was developed in \cite{ImbertMonneau, LionsSouganidis2016}, and the selection of solutions at the junction by vanishing viscosity was studied in \cite{CamilliMarchiSchieborn2013}. Weakly coupled systems of eikonal equations arise, for example, in path planning with random switches between modes of motion \cite{GeeVladimirsky2022}. In this paper, we extend the network setting from a single equation to such systems, and we develop a numerical scheme for them at a junction. We consider a star graph, that is, $M$ branches that meet at a single junction, and on each branch a system of two equations of the form $|u_i'| + C_i (u_i - u_{3-i}) = f_i$, $i = 1, 2$, with positive coupling constants $C_i$ and nonnegative right-hand sides $f_i$.

We do not impose a junction condition. Instead, we ask for which boundary values at the outer vertices the solution on each branch, computed from its own boundary vertex toward the junction, is continuous at the junction and satisfies the monotonicity conditions of the scheme. We call such boundary values admissible.

Our formulation differs from the usual treatment of Hamilton--Jacobi equations on networks, where the boundary data are prescribed and a junction condition is imposed in order to select a solution. Here we consider a different, inverse compatibility problem. No condition is prescribed at the junction. Instead, on each branch we propagate the solution from the outer boundary vertex toward the junction and ask which choices of the outer boundary data produce branchwise solutions that agree there. Thus, continuity at the junction is not used as a selection condition for a problem with fixed boundary data, but rather as a constraint determining the admissible boundary data. Our purpose is therefore not to introduce an alternative notion of solution at a junction, but to characterize explicitly the boundary data for which the independently propagated branchwise solutions are compatible. This viewpoint leads, in the discrete problem considered below, to an explicit affine representation of the values at the junction and ultimately to a finite-dimensional characterization of the set of admissible boundary data.

On each branch, we approximate the right-hand sides by their cell averages. On a cell, each equation of the system then has a constant right-hand side, and its solution is an explicit combination of exponentials. Coupling the two equations through cell averages gives an iteration, which we write in matrix form with bidiagonal matrices; its limits are available in closed form (Algorithm~\ref{alg:ch_2_sys-genC-nonhom-limit}). Monotone schemes of this type go back to Crandall and Lions \cite{CrandallLions1984}; each single equation can also be solved by the fast marching method \cite{Sethian}. For a single equation on a network or a junction, approximation schemes and their convergence were studied in \cite{CamilliFestaSchieborn2013, CostesequeLebacqueMonneau2015}, and error estimates were obtained in \cite{GuerandKoumaiha2019}.

The main results concern the star graph. The values of the discrete solution at the junction are $\mathfrak{a}_+^{(m)} + \mathfrak{D}_+^{(m)} \mathfrak{g}_+^{(m)}$, where $\mathfrak{g}_+^{(m)}$ is the pair of boundary values on the branch $J_m$. The entries of the $2 \times 2$ matrix $\mathfrak{D}_+^{(m)}$ are positive and its rows sum to one (Lemma~\ref{lem:bc_positive} and Theorem~\ref{thrm:app_C_mathfrakD_matrix_rank_one}). Its determinant is a product of explicit factors, one for each cell, and $\mathfrak{D}_+^{(m)}$ is singular if and only if $\mathfrak{B}(C_{1, m} h_j) + \mathfrak{B}(C_{2, m} h_j) = 1$ for some cell, where $\mathfrak{B}$ is the Bernoulli function (Theorem~\ref{thrm:det_D_product}). When all these matrices are nonsingular, continuity at the junction determines the boundary values on every branch from those on the first branch (Lemma~\ref{lem:rev_continuity}). The monotonicity conditions depend on the boundary values only through the difference of the two components (Lemma~\ref{lem:rev_strip}), and the admissible boundary values form a convex polygon in the plane: a strip parallel to $(1, 1)^T$ cut by positivity conditions (Theorem~\ref{thm:rev_admissible}). Finding an admissible pair is a linear feasibility problem in two variables (Algorithm~\ref{alg:network_system_rev}). When a matrix $\mathfrak{D}_+^{(m)}$ is singular, the difference of the two components at the junction is fixed by the data on that branch, and continuity requires an additional condition on the boundary values of the other branches. In our experiments, the discrete solutions converge at second order in the mesh size, both on an interval and on a star graph.

The rest of the paper is organized as follows. In \S\ref{sec:prelim}, we state the problems on a star network and on an interval, and we introduce the discretization. In \S\ref{sec:schemes}, we present the schemes for approximating the solution to a coupled system of Hamilton-Jacobi equations, together with the algorithms that compute its solution; we show monotonicity of the iteration on an interval, which carries over branch by branch to a star graph. In \S\ref{chapter3:network}, we compute the values at the junction obtained from the boundary data on each branch, we study the continuity of the solutions at the junction, and we obtain explicit conditions on the boundary data under which a continuous solution exists. We also design an algorithm for the computation of admissible boundary values. In \S\ref{sec:experiments}, we present numerical experiments on an interval and on star graphs.

\section{Problem formulation and discretization}\label{sec:prelim}

\subsection{Hamilton-Jacobi systems on a star graph}\label{sec:prelim_network}
We consider a network with a single junction and $M$ edges (branches) intersecting at it, that is, a star graph, which we illustrate in Figure \ref{fig:ch_2_junction}. We denote the branches by $J_m$ for $m = 1, \ldots, M$, and we label the junction the origin, that is, $\cap_{m = 1}^M J_m = \{ \bm 0 \}$. Let $\bm x_m : I_m \to J_m$ denote the arc length parameterization of the branch $J_m$, where $I_m = [0, \ell_m]$ with $\ell_m$ being the length of the branch $J_m$, and $\bm x_m(0) = \bm 0$. We consider the problem
\begin{equation}\label{eq:net_system}
    \begin{split}
        |\widetilde{u}_{1, m}'| + C_{1, m} (\widetilde{u}_{1, m} - \widetilde{u}_{2, m}) &= f_{1, m}(s), \quad s \in I_m, \\
        |\widetilde{u}_{2, m}'| + C_{2, m} (\widetilde{u}_{2, m} - \widetilde{u}_{1, m}) &= f_{2, m}(s), \quad s \in I_m, \\
        \widetilde{u}_{1, m}(\ell_m) &= g_{1, m}(\ell_m), \quad \widetilde{u}_{2, m}(\ell_m) = g_{2, m}(\ell_m),
    \end{split}
\end{equation}
for $m = 1, \ldots, M$, where $C_{j, m}$ are {positive} constants, and $f_{j, m}$ are continuous functions for $j = 1, 2$. The unknowns are two functions $u_j : \cup_{m=1}^M J_m \to \mathbb{R}$, $j = 1, 2$. We denote by $u_{j, m} : J_m \to \mathbb{R}$ the restriction of $u_j$ to the branch $J_m$, and by $\widetilde{u}_{j,m}(s) = u_{j,m}(\bm x_m(s))$, $s \in I_m$, these restrictions as functions of the arc length. The boundary values are prescribed at $\bm x_m(\ell_m)$ for $m = 1, \ldots, M$. At the junction, we do not prescribe a boundary value. Rather, we prescribe the continuity condition
\begin{equation}
    u_{j, p}(\bm 0) = u_{j, q}(\bm 0),
\end{equation}
for $j = 1, 2$ and $p, q = 1, \ldots, M$. We consider propagating data from each of the boundaries.

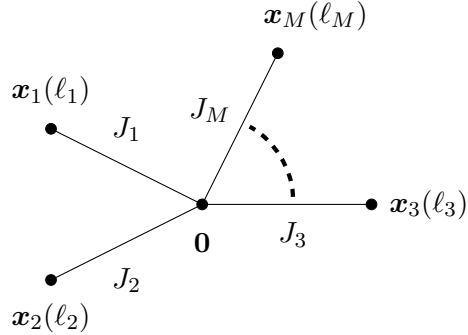
\begin{figure}[ht]
    \centering
    \begin{tikzpicture}
        \draw (-2, 1) -- (0, 0);
        \draw (-2, -1) -- (0, 0);
        \draw (0, 0) -- (2.24, 0);
        \draw (0, 0) -- (1, 2);

        \draw[ultra thick, dashed, shorten <=1mm, shorten >=1mm] (1.2, 0) arc[start angle=0, end angle=63.4, radius=1.2cm];

        \node at (-1, 1) {$J_1$};
        \node at (-1, -1) {$J_2$};
        \node at (1.2, -0.4) {$J_3$};
        \node at (0.1, 1.25) {$J_M$};

        \node at (0, -0.5) {$\bm 0$};
        \filldraw[black] (0, 0) circle (2pt);

        \node at (-2, 1.5) {$\bm x_1(\ell_1)$};
        \filldraw[black] (-2, 1) circle (2pt);
        \node at (-2, -1.5) {$\bm x_2(\ell_2)$};
        \filldraw[black] (-2, -1) circle (2pt);
        \node at (3, 0) {$\bm x_3(\ell_3)$};
        \filldraw[black] (2.24, 0) circle (2pt);
        \node at (1.5, 2.5) {$\bm x_M(\ell_M)$};
        \filldraw[black] (1, 2) circle (2pt);
    \end{tikzpicture}
    \caption{A junction with branches}
    \label{fig:ch_2_junction}
\end{figure}


We study the following problem. The boundary values $g_{i, m}(\ell_m)$ in \eqref{eq:net_system} are unknown. We seek the boundary values for which the solution on each branch $J_m$ is determined by the boundary data at its own boundary vertex $\bm x_m(\ell_m)$ all the way to the junction, and for which the resulting functions satisfy the continuity conditions at the junction. We call such boundary values admissible. For admissible boundary values, the junction is the point where the characteristics from all boundary vertices meet.

\subsection{The system on a branch}\label{chapter2:1D}
On each branch $J_m$, the parameterization $\bm x_m$ turns \eqref{eq:net_system} into a system on the interval $I_m$. We therefore solve, on every branch, a system of the following form:
\begin{equation}\label{eq:ch_2_1D_gen_system}
    \begin{split}
        |u_1'| + C_1 (u_1 - u_2) &= f_1(x), \quad x \in (a, b), \\
        |u_2'| + C_2 (u_2 - u_1) &= f_2(x), \quad x \in (a, b), \\
        u_1 (x) &= g_1 (x), \quad u_2 (x) = g_2 (x), \quad x \in \{a, b\},
    \end{split}
\end{equation}
where $a < b$; $C_i$ is {a positive} constant for $i = 1, 2$; $g_i$ are given for $i = 1, 2$; $f_i$ is known for $i = 1, 2$; and $f_i(x) \geq 0$ for $i = 1, 2$. 
The above system consists of two weakly coupled Hamilton-Jacobi equations, where the Hamiltonians are
\begin{equation}
    H_i(u_i', x) = |u_i'| - f_i(x).
\end{equation}
Our approach is to construct an iterative scheme that converges to the solution of the system~\eqref{eq:ch_2_1D_gen_system}.

\subsection{Solving a single equation with constant right-hand side}\label{sec:single_const}
The discretization of \eqref{eq:ch_2_1D_gen_system} replaces the right-hand sides by constants on the cells of a partition. The solution of the system on an interval therefore requires the viscosity solution of a single equation with constant right-hand side on an interval $(a, b)$:
\begin{equation}\label{eq:single_const}
    \begin{aligned}
        &|u'(x)| + C_1 u(x) = C, \quad x \in (a, b), \\
        &u(a) = g_a, \qquad u(b) = g_b,
    \end{aligned}
\end{equation}
where $C_1 > 0$ and $C \geq 0$ are constants. Suppose that $C_1 g_a \leq C$ and $C_1 g_b \leq C$. Let
\begin{equation}\label{eq:single_const_vw}
    \begin{aligned}
        v(x) &= \frac{C}{C_1} + \left( g_a - \frac{C}{C_1} \right) \exp(-C_1 (x - a)), \\
        w(x) &= \frac{C}{C_1} + \left( g_b - \frac{C}{C_1} \right) \exp(-C_1 (b - x)).
    \end{aligned}
\end{equation}
If, in addition, $v(b) \geq g_b$ and $w(a) \geq g_a$, then the viscosity solution of \eqref{eq:single_const} is their minimum,
\begin{equation}\label{eq:ch_2_1D_1_eq_C1_C2_sol}
    \begin{aligned}
        u^*(x) &= \min\left\{ v(x), w(x) \right\} \\
        &= \frac{C}{C_1} + \min\left\{ \left( g_a - \frac{C}{C_1} \right) \exp(-C_1 (x - a)), \left( g_b - \frac{C}{C_1} \right) \exp(-C_1 (b - x)) \right\}.
    \end{aligned}
\end{equation}
We use this solution as a model to construct an iterative scheme 
and numerically solve the problem~\eqref{eq:ch_2_1D_gen_system}.
From now on, without loss of generality, we take $a = -1$ and $b = 1$.

\subsection{Partitions, cell averages, and matrix notation}\label{sec:discretization}
On $(-1, 1)$, we use the following notation.
\begin{itemize}
    \item We set $T_j=(x_{j-1},x_{j})$, $x_0=-1$, $x_N=1$, $x_j=x_{j-1}+h_j$, $j=1:N$, and $\bm h=\{h_j\}_{j=1}^N$.
    \item For a right-hand side $f$, we denote $F_j=\frac{1}{|T_j|}\int_{T_j}f(t) dt$ and $\bm F=\{F_j\}_{j=1}^N$. For right-hand sides $f_i$, $i = 1, 2$, we denote $F_{i,j}=\frac{1}{|T_j|}\int_{T_j}f_i(t) dt$ and $\bm F_i=\{F_{i,j}\}_{j=1}^N$.
\end{itemize}
We define the $N \times N$ matrices
\begin{equation}\label{eq:ch_2_Apm}
    \begin{aligned}
        A_- &:= I - \sum_{j=2}^{N} \exp(-C_1 h_j)\, \bm e_j \bm e_{j-1}^T, &\qquad P_- &:= \sum_{j=2}^{N} \bm e_j \bm e_{j-1}^T,\\
        A_+ &:= I - \sum_{j=1}^{N-1} \exp(-C_1 h_j)\, \bm e_j \bm e_{j+1}^T, &\qquad P_+ &:= \sum_{j=1}^{N-1} \bm e_j \bm e_{j+1}^T,
    \end{aligned}
\end{equation}
where{, since $C_1 > 0$,} $A_\pm$ are diagonally dominant two-diagonal M-matrices and $P_\pm$ are the shift matrices.
For a constant $C_i$, we denote by $A_{i, \pm}$ the matrices $A_\pm$ with $C_1$ replaced by $C_i$. These matrices have the form $I - X$, where $X$ is strictly lower triangular for $A_{i, -}$ and strictly upper triangular for $A_{i, +}$. Applying $(I - X)^{-1} = \sum_{\ell = 0}^{N-1} X^\ell$, we get
\begin{equation}
    \left( A_{i, -}^{-1} \right)_{j \ell} = \exp\left( -C_i \sum_{s = \ell+1}^{j} h_s \right) \quad \text{for } \ell \leq j, \qquad \left( A_{i, +}^{-1} \right)_{j \ell} = \exp\left( -C_i \sum_{s = j}^{\ell-1} h_s \right) \quad \text{for } \ell \geq j,
\end{equation}
and all other entries are zero.
With the well-known Bernoulli function $\mathfrak{B}$, we set
\begin{equation}\label{eq:bernoulli}
    \mathfrak{B}(t)=\begin{cases} 
        \frac{t}{e^t-1} & t\neq 0,\\
        1 & t=0,
    \end{cases}
    \qquad
    \beta_{i, j} = 1-\frac{1}{\mathfrak{B}(-C_ih_j)}, \quad \bm \beta_i = \{ \beta_{i,j} \}_{j=1}^N.
\end{equation}
Following the direction of the parameterization, let $\bm h^{(m)}$ be the vector $\bm h$ on the branch $J_m$, where we index the components of $\bm h^{(m)}$ by $1, \ldots, N_m$. We denote by $T_j^{(m)}$, $j = 1:N_m$, the corresponding cells of $I_m$, and by $\bm F_i^{(m)}$ the vectors of averages on these cells of the right-hand sides $f_{i, m}$ in \eqref{eq:net_system}. We denote by $A_{i, \pm}^{(m)}$ the matrices $A_\pm$ in \eqref{eq:ch_2_Apm} with $C_1$, $\bm h$, and $N$ replaced by $C_{i, m}$, $\bm h^{(m)}$, and $N_m$.

\subsection{The iteration for the system}\label{sec:system_iteration}
We propose an iteration to solve the system. We take a piecewise-constant approximation to the system given in equation \eqref{eq:ch_2_1D_gen_system} with general right-hand side $f_1(x), f_2(x) \geq 0$ for $x \in (-1, 1)$. We start by moving a piecewise-constant approximation of $u_2$ to the right-hand side of the first equation. We then solve the first equation for $u_1$. Approximating the updated $u_1$ by a piecewise constant and moving it to the right-hand side of the second equation, we then solve the second equation for $u_2$. We use known (model) solutions in simple cases to define the iterates.

We use the partition $\{T_j\}_{j=1}^N$ and the vector $\bm h$ defined in \S\ref{sec:discretization}, and we consider the following iteration:
\begin{itemize}
    \item The functions $v_1$ and $v_2$ correspond to the solution from the left, and the functions $w_1$ and $w_2$ correspond to the solution from the right.
    \item We take as boundary conditions for $k=1,2,\ldots$:
    \begin{equation}
        \begin{aligned}
            &v_1^{(k)}(x_0) = g_1(x_0), \quad v_2^{(k)}(x_0) = g_2(x_0), \\
            &w_1^{(k)}(x_N) = g_1(x_N), \quad w_2^{(k)}(x_N) = g_2(x_N).
        \end{aligned}
    \end{equation}
    \item We take any initial guesses satisfying these boundary conditions.
    \item Starting from the left and the right end of the interval, respectively, we define $v_{1,2}^{(k)}$ and $w_{1,2}^{(k)}$ as follows:
    \begin{equation}
        \begin{aligned}
            &\left|\frac{d v_1^{(k)}}{dx}\right|+C_1v_1^{(k)} =F_{1,j}+C_1b_{j,k-1}, \quad \mbox{on}\quad T_j,\quad 
            v_{1}^{(k)}(x_{j-1}) \quad\mbox{given},\\
            &\left|\frac{d v_2^{(k)}}{dx}\right|+C_2v_2^{(k)} =F_{2,j}+C_2a_{j,k}, \quad\mbox{on}\quad T_j,\quad  v_{2}^{(k)}(x_{j-1}) \quad\mbox{given},\\
            &\left|\frac{d w_1^{(k)}}{dx}\right|+C_1w_1^{(k)} =F_{1,j}+C_1d_{j,k-1}, \quad \mbox{on}\quad T_j,\quad 
            w_{1}^{(k)}(x_{j}) \quad\mbox{given},\\
            &\left|\frac{d w_2^{(k)}}{dx}\right|+C_2w_2^{(k)} =F_{2,j}+C_2c_{j,k} \quad\mbox{on}\quad T_j,\quad  w_{2}^{(k)}(x_{j}) \quad\mbox{given}.
        \end{aligned}
    \end{equation}
    \item Here, we have denoted 
    \begin{equation}\label{eq:ch_2_abcd-def}
        \begin{aligned}
            &a_{j,k}=\frac{1}{|T_j|}\int_{T_j} v_1^{(k)},\quad 
            b_{j,k}=\frac{1}{|T_j|}\int_{T_j} v_2^{(k)},\\    
            &c_{j,k}=\frac{1}{|T_j|}\int_{T_j} w_1^{(k)},\quad    
            d_{j,k}=\frac{1}{|T_j|}\int_{T_j} w_2^{(k)}.\\    
        \end{aligned}
    \end{equation}
\end{itemize}

\section{Discrete schemes and their solution}\label{sec:schemes}

\subsection{Discretization of a system and its iterative solution}\label{sec:ch_2_1D_system}\label{subsec:ch_2_1D_system_scheme}
By \S\ref{sec:single_const}, we know the solutions to each of the differential problems defining the iteration in \S\ref{sec:system_iteration}. On $T_j$, the function $v_1^{(k)}$ is the function $v$ in \eqref{eq:single_const_vw} with $(a, b) = T_j$, $C = F_{1,j} + C_1 b_{j,k-1}$, and $g_a = v_1^{(k)}(x_{j-1})$, and $w_1^{(k)}$ is the function $w$ in \eqref{eq:single_const_vw} with $(a, b) = T_j$, $C = F_{1,j} + C_1 d_{j,k-1}$, and $g_b = w_1^{(k)}(x_j)$. The functions $v_2^{(k)}$ and $w_2^{(k)}$ are obtained in the same way, with $C_1$ replaced by $C_2$ and with $C = F_{2,j} + C_2 a_{j,k}$ and $C = F_{2,j} + C_2 c_{j,k}$, respectively. Averaging these functions over $T_j$ and using $\left(1-\exp(-C_ih_j)\right)/(C_ih_j) = 1/\mathfrak{B}(-C_ih_j) = 1-\beta_{i,j}$ with $\beta_{i, j}$ given in \eqref{eq:bernoulli} then show that
\begin{equation}\label{eq:ch_2_abcd-comp-genC}
    \begin{aligned}
        a_{j,k} & =\beta_{1, j} \frac{\left(F_{1,j}+C_1b_{j,k-1}\right)}{C_1}+ (1-\beta_{1, j})v_{1}^{(k)}(x_{j-1}),\\
        b_{j,k} & =\beta_{2, j} \frac{\left(F_{2,j}+C_2a_{j,k}\right)}{C_2} + (1-\beta_{2, j})v_{2}^{(k)}(x_{j-1}),\\
        c_{j,k} & =\beta_{1, j}\frac{\left(F_{1,j}+C_1d_{j,k-1}\right)}{C_1} + (1-\beta_{1, j})w_{1}^{(k)}(x_{j}),\\
        d_{j,k} & =\beta_{2, j}\frac{\left(F_{2,j}+C_2c_{j,k}\right)}{C_2} + (1-\beta_{2, j})w_{2}^{(k)}(x_{j}).
    \end{aligned}
\end{equation}
Evaluating these functions at $x = x_j$ or $x = x_{j-1}$, as appropriate, we have
\begin{equation}\label{eq:ch_2_differences_vw}
    \begin{aligned}
        v^{(k)}_1(x_j)-\exp(-C_1h_j)v^{(k)}_1(x_{j-1}) &= \frac{\left(F_{1,j}+C_1b_{j,k-1}\right)}{C_1}(1-\exp(-C_1h_j)), \\
        v^{(k)}_2(x_j)-\exp(-C_2h_j)v^{(k)}_2(x_{j-1}) &= \frac{\left(F_{2,j}+C_2a_{j,k}\right)}{C_2}(1-\exp(-C_2h_j)), \\
        w^{(k)}_1(x_{j-1})-\exp(-C_1h_j)w^{(k)}_1(x_{j}) &= \frac{\left(F_{1,j}+C_1d_{j,k-1}\right)}{C_1}(1-\exp(-C_1h_j)), \\
        w^{(k)}_2(x_{j-1})-\exp(-C_2h_j)w^{(k)}_2(x_{j}) &= \frac{\left(F_{2,j}+C_2c_{j,k}\right)}{C_2}(1-\exp(-C_2h_j)). \\ 
    \end{aligned}
\end{equation}

\subsection{On the convergence of the iterations}\label{sec:ch_2_specific_case}\label{subsec:ch_2_1D_system_convergence}
We can compute $v^{(k)}_{1, 2}(x_j)$ for $j = 1:N$ and $w^{(k)}_{1, 2}(x_j)$ for $j=0:(N-1)$ by solving lower (upper) triangular systems with the matrices $A_{i, \pm}$ from \S\ref{sec:discretization}. Furthermore, we denote the vectors
\begin{equation}
    \begin{aligned}
        \bm a^{(k)}=\{a_{j,k}\}_{j=1}^{N}, &\quad \bm b^{(k)}=\{b_{j,k}\}_{j=1}^{N}, \\
        \bm c^{(k)}=\{c_{j,k}\}_{j=1}^{N}, &\quad \bm d^{(k)}=\{d_{j,k}\}_{j=1}^{N}, \\
        \bm F_i=\{F_{i,j}\}_{j=1}^N, &\quad \bm \beta_i = \{ \beta_{i,j} \}_{j=1}^N, \quad \bm \phi_i=\operatorname{diag}(\bm \beta_i) \bm F_i, \\
        \bm v_{1,2}^{(k)}=\{v_{1,2}^{(k)}(x_j)\}_{j=1}^{N}, &\quad \bm w_{1,2}^{(k)}=\{w_{1,2}^{(k)}(x_j)\}_{j=0}^{N-1}.
    \end{aligned}
\end{equation}
Writing the relations \eqref{eq:ch_2_differences_vw} and \eqref{eq:ch_2_abcd-comp-genC} in terms of these vectors gives the iteration in matrix form, which we now state.

To simplify the formulas for the iterates and the limits, we define the matrices
\begin{equation}\label{eq:ch_2_M}
    M_{i, \pm} = A_{i, \pm}^{-1} (I - \operatorname{diag}(\exp(-C_i \bm h))).
\end{equation}
With this, we can write $Z_{i, \pm}$ as
\begin{equation}\label{eq:ch_2_Z_i_pm}
    Z_{i, \pm} = \operatorname{diag}(\bm \beta_i) + (I - \operatorname{diag}(\bm \beta_i)) P_{\pm} M_{i, \pm}.
\end{equation}
Further, define the following vectors:
\begin{equation}\label{eq:ch_2_omega}
    \begin{split}
        \bm \omega_{i, -} &= \frac{1}{C_i} M_{i, -} \bm F_i + g_i(x_0) \exp(-C_i h_1) A_{i, -}^{-1} \bm e_1, \\
        \bm \omega_{i, +} &= \frac{1}{C_i} M_{i, +} \bm F_i + g_i(x_N) \exp(-C_i h_N) A_{i, +}^{-1} \bm e_N,
    \end{split}
\end{equation}
and
\begin{equation}\label{eq:ch_2_a0b0c0d0}
    \begin{aligned}
        \bm a_1 &= \frac{1}{C_1} \bm \phi_1 + (I - \operatorname{diag}(\bm \beta_1)) \left( g_1(x_0) \bm e_1 + P_- \bm \omega_{1, -} \right), \\
        \bm b_1 &= \frac{1}{C_2} \bm \phi_2 + (I - \operatorname{diag}(\bm \beta_2)) \left( g_2(x_0) \bm e_1 + P_- \bm \omega_{2, -} \right), \\
        \bm c_1 &= \frac{1}{C_1} \bm \phi_1 + (I - \operatorname{diag}(\bm \beta_1)) \left( g_1(x_N) \bm e_N + P_+ \bm \omega_{1, +} \right), \\
        \bm d_1 &= \frac{1}{C_2} \bm \phi_2 + (I - \operatorname{diag}(\bm \beta_2)) \left( g_2(x_N) \bm e_N + P_+ \bm \omega_{2, +} \right).
    \end{aligned}
\end{equation}
These quantities are independent of the iteration $k$. For the set of values propagating from the left, the iterates are
\begin{equation}
    \begin{split}
        \bm v_1^{(k)} &= M_{1, -} \bm b^{(k-1)} + \bm \omega_{1, -}, \\
        \bm a^{(k)} &= \operatorname{diag}(\bm \beta_1) \bm b^{(k-1)} + (I - \operatorname{diag}(\bm \beta_1)) \left( g_1(x_0) \bm e_1 + P_- \bm v_1^{(k)} \right) + \frac{1}{C_1} \bm \phi_1, \\
        \bm v_2^{(k)} &= M_{2, -} \bm a^{(k)} + \bm \omega_{2, -}, \\
        \bm b^{(k)} &= \operatorname{diag}(\bm \beta_2) \bm a^{(k)} + (I - \operatorname{diag}(\bm \beta_2)) \left( g_2(x_0) \bm e_1 + P_- \bm v_2^{(k)} \right) + \frac{1}{C_2} \bm \phi_2.
    \end{split}
\end{equation}
Combining the above relations, we get
\begin{equation}
    \bm a^{(k)} = Z_{1, -} \bm b^{(k-1)} + \bm a_1, \qquad
    \bm b^{(k)} = Z_{2, -} \bm a^{(k)} + \bm b_1.
\end{equation}
We can rearrange these to get
\begin{equation}
    \begin{split}
        \bm a^{(k)} &= Z_{1, -} Z_{2, -} \bm a^{(k-1)} + Z_{1, -} \bm b_1 + \bm a_1, \\
        \bm b^{(k)} &= Z_{2, -} Z_{1, -} \bm b^{(k-1)} + Z_{2, -} \bm a_1 + \bm b_1.
    \end{split}
\end{equation}
By \eqref{eq:ch_2_Z_i_pm}, the matrices $Z_{i, -}$ are lower triangular and the matrices $Z_{i, +}$ are upper triangular. {Since} $C_i > 0$, we have $\mathfrak{B}(-C_i h_j) > 1$, and therefore
\begin{equation}\label{eq:ch_2_diag-eig}
    \begin{aligned}
        & \left( Z_{i, \pm} \right)_{jj} = \beta_{i, j} = 1 - \frac{1}{\mathfrak{B}(-C_i h_j)} \in (0, 1), \\
        & \rho(Z_{2, \pm} Z_{1, \pm}) = \rho(Z_{1, \pm} Z_{2, \pm}) = \max_{j = 1:N} \beta_{1, j} \beta_{2, j} < 1.
    \end{aligned}
\end{equation}
Hence the iterations converge for every vector $\bm h$ defined in \S\ref{sec:discretization}, and the limits are independent of the initial values $\bm b^{(0)}$ and $\bm d^{(0)}$.

In the limit, we obtain the following:
\begin{equation}\label{eq:ch_2_v1av2b_limits_g}
    \begin{split}
        \bm v_1^{(\infty)} &= M_{1, -} \left( I - Z_{2, -} Z_{1, -} \right)^{-1} \left( Z_{2, -} \bm a_1 + \bm b_1 \right) + \bm \omega_{1, -}, \\
        \bm v_2^{(\infty)} &= M_{2, -} \left( I - Z_{1, -} Z_{2, -} \right)^{-1} \left( Z_{1, -} \bm b_1 + \bm a_1 \right) + \bm \omega_{2, -}.
    \end{split}
\end{equation}
For the set of values propagating from the right, the iterates are
\begin{equation}
    \begin{split}
        \bm w_1^{(k)} &= M_{1, +} \bm d^{(k-1)} + \bm \omega_{1, +}, \\
        \bm c^{(k)} &= \operatorname{diag}(\bm \beta_1) \bm d^{(k-1)} + (I - \operatorname{diag}(\bm \beta_1)) \left( g_1(x_N) \bm e_N + P_+ \bm w_1^{(k)} \right) + \frac{1}{C_1} \bm \phi_1, \\
        \bm w_2^{(k)} &= M_{2, +} \bm c^{(k)} + \bm \omega_{2, +}, \\
        \bm d^{(k)} &= \operatorname{diag}(\bm \beta_2) \bm c^{(k)} + (I - \operatorname{diag}(\bm \beta_2)) \left( g_2(x_N) \bm e_N + P_+ \bm w_2^{(k)} \right) + \frac{1}{C_2} \bm \phi_2.
    \end{split}
\end{equation}
By similar computations as given above, in the limit, we obtain the following:
\begin{equation}\label{eq:ch_2_w1cw2d_limits_g}
    \begin{split}
        \bm w_1^{(\infty)} &= M_{1, +} \left( I - Z_{2, +} Z_{1, +} \right)^{-1} \left( Z_{2, +} \bm c_1 + \bm d_1 \right) + \bm \omega_{1, +}, \\
        \bm w_2^{(\infty)} &= M_{2, +} \left( I - Z_{1, +} Z_{2, +} \right)^{-1} \left( Z_{1, +} \bm d_1 + \bm c_1 \right) + \bm \omega_{2, +}.
    \end{split}
\end{equation}
Since we know the limits of the iterates, we can compute these limits directly, as in Algorithm \ref{alg:ch_2_sys-genC-nonhom-limit}.
\begin{algorithm}[H]
    \caption{Limits of the iterations from the left and from the right for a system}\label{alg:ch_2_sys-genC-nonhom-limit}
    \begin{algorithmic}
    \Function{CalculateLimits}{}
    \State Compute $\bm{\beta}_i=\{\beta_{i,j}\}_{j=1}^N$, $\bm F_i$, and $\bm \phi_i$;
    \State Compute $M_{i, \pm}$, $\bm \omega_{i, \pm}$, $\bm a_1$, $\bm b_1$, $\bm c_1$, $\bm d_1$;
    \State $\bm v_1^{(\infty)} = M_{1, -} \left( I - Z_{2, -} Z_{1, -} \right)^{-1} \left( Z_{2, -} \bm a_1 + \bm b_1 \right) + \bm \omega_{1, -}$;
    \State $\bm w_1^{(\infty)} = M_{1, +} \left( I - Z_{2, +} Z_{1, +} \right)^{-1} \left( Z_{2, +} \bm c_1 + \bm d_1 \right) + \bm \omega_{1, +}$;
    \State $\bm v_2^{(\infty)} = M_{2, -} \left( I - Z_{1, -} Z_{2, -} \right)^{-1} \left( Z_{1, -} \bm b_1 + \bm a_1 \right) + \bm \omega_{2, -}$;
    \State $\bm w_2^{(\infty)} = M_{2, +} \left( I - Z_{1, +} Z_{2, +} \right)^{-1} \left( Z_{1, +} \bm d_1 + \bm c_1 \right) + \bm \omega_{2, +}$;
    \EndFunction
    \end{algorithmic}
\end{algorithm}

The limits in Algorithm \ref{alg:ch_2_sys-genC-nonhom-limit} are used in \S\ref{sec:junction_values}, where only the iteration from one end of the interval is needed. To compute the solution of \eqref{eq:ch_2_1D_gen_system} on $(-1, 1)$, we couple each equation to the other through the solution itself. We denote by $\overline{\bm u}_i = \left\{ \left( u_i(x_{j-1}) + u_i(x_j) \right)/2 \right\}_{j=1}^N$ the vector of cell averages of $u_i$ computed by the trapezoidal rule, and we iterate as in Algorithm \ref{alg:ch_2_sys_general}, where each step solves a single equation with piecewise-constant right-hand side, from both ends of the interval, with $M_{i, \pm}$ in \eqref{eq:ch_2_M} and $\bm \omega_{i, \pm}$ in \eqref{eq:ch_2_omega}. We do not prove convergence of this iteration.
\begin{algorithm}[H]
    \caption{Iterative solution of the system \eqref{eq:ch_2_1D_gen_system}}\label{alg:ch_2_sys_general}
    \begin{algorithmic}
    \Function{SolveSystem}{$\bm u_1^{(0)}, \bm u_2^{(0)}, \varepsilon$}
    \For{$k = 1, 2, \ldots$}
    \State $\bm v_1 = M_{1, -} \overline{\bm u}_2^{(k-1)} + \bm \omega_{1, -}$, \quad $\bm w_1 = M_{1, +} \overline{\bm u}_2^{(k-1)} + \bm \omega_{1, +}$;
    \State $\bm u_1^{(k)} = \min\{ [g_1(x_0); \bm v_1], [\bm w_1; g_1(x_N)] \}$;
    \State $\bm v_2 = M_{2, -} \overline{\bm u}_1^{(k)} + \bm \omega_{2, -}$, \quad $\bm w_2 = M_{2, +} \overline{\bm u}_1^{(k)} + \bm \omega_{2, +}$;
    \State $\bm u_2^{(k)} = \min\{ [g_2(x_0); \bm v_2], [\bm w_2; g_2(x_N)] \}$;
    \If{$\max_{i = 1, 2} \left\| \bm u_i^{(k)} - \bm u_i^{(k-1)} \right\|_\infty \leq \varepsilon$}
        \State \Return $\bm u_1^{(k)}$, $\bm u_2^{(k)}$;
    \EndIf
    \EndFor
    \EndFunction
    \end{algorithmic}
\end{algorithm}

\begin{remark}\label{rmk:fmm}
In Algorithm \ref{alg:ch_2_sys_general}, the solution of each single equation can also be computed by the fast marching method \cite{Sethian}. Starting from both endpoints, at each step one accepts the node with the smaller tentative value, computed from its accepted neighbor by the solution \eqref{eq:ch_2_1D_1_eq_C1_C2_sol} on the cell between them. When $\bm v_i$ is nondecreasing and $\bm w_i$ is nonincreasing, this gives the same vector $\bm u_i^{(k)}$ as Algorithm \ref{alg:ch_2_sys_general}.
\end{remark}

\subsection{Monotonicity of the iterates}\label{sec:analysis}\label{sec:analysis_system}
\begin{lemma}\label{lem:ch_2_1D_2_eq_monotone_scheme}
Suppose that
\begin{equation}
    \begin{aligned}
        F_{1, j} + C_1 b_{j, k-1} &\geq C_1 v_1^{(k)}(x_{j-1}), &\qquad F_{2, j} + C_2 a_{j, k} &\geq C_2 v_2^{(k)}(x_{j-1}), \\
        F_{1, j} + C_1 d_{j, k-1} &\geq C_1 w_1^{(k)}(x_j), &\qquad F_{2, j} + C_2 c_{j, k} &\geq C_2 w_2^{(k)}(x_j),
    \end{aligned}
\end{equation}
for $j = 1:N$. Then
\begin{equation}
    v_i^{(k)}(x_{j-1}) \leq v_i^{(k)}(x_j) \quad \text{and} \quad w_i^{(k)}(x_j) \leq w_i^{(k)}(x_{j-1}), \quad j = 1:N, \ i = 1, 2.
\end{equation}
\end{lemma}
\begin{proof}
We use the formulas in equation \eqref{eq:ch_2_differences_vw}. For the data emanating from the left, we have $v_1^{(k)}(x_0) = g_1(x_0)$, $v_2^{(k)}(x_0) = g_2(x_0)$, and
\begin{equation}
    \begin{split}
        v^{(k)}_1(x_j) &= \exp(-C_1h_j)v^{(k)}_1(x_{j-1}) + \frac{\left(F_{1,j}+C_1b_{j,k-1}\right)}{C_1}(1-\exp(-C_1h_j)), \\
        v^{(k)}_2(x_j) &= \exp(-C_2h_j)v^{(k)}_2(x_{j-1}) + \frac{\left(F_{2,j}+C_2a_{j,k}\right)}{C_2}(1-\exp(-C_2h_j)).
    \end{split}
\end{equation}
Applying the first two conditions{, divided by $C_1 > 0$ and $C_2 > 0$, respectively,} we then estimate
\begin{equation}
    \begin{split}
        v_1^{(k)}(x_j) &\geq \exp(-C_1h_j)v^{(k)}_1(x_{j-1}) + v_1^{(k)}(x_{j-1})(1-\exp(-C_1h_j)) = v_1^{(k)}(x_{j-1}), \\
        v_2^{(k)}(x_j) &\geq \exp(-C_2h_j)v^{(k)}_2(x_{j-1}) + v_2^{(k)}(x_{j-1})(1-\exp(-C_2h_j)) = v_2^{(k)}(x_{j-1}).
    \end{split}
\end{equation}
For the data emanating from the right, we have $w_1^{(k)}(x_N) = g_1(x_N)$, $w_2^{(k)}(x_N) = g_2(x_N)$, and
\begin{equation}
    \begin{split}
        w_1^{(k)}(x_{j-1}) &= \exp(-C_1h_j)w^{(k)}_1(x_{j}) + \frac{\left(F_{1,j}+C_1d_{j,k-1}\right)}{C_1}(1-\exp(-C_1h_j)), \\
        w^{(k)}_2(x_{j-1}) &= \exp(-C_2h_j)w^{(k)}_2(x_{j}) + \frac{\left(F_{2,j}+C_2c_{j,k}\right)}{C_2}(1-\exp(-C_2h_j)).
    \end{split}
\end{equation}
Applying the last two conditions{, divided by $C_1 > 0$ and $C_2 > 0$, respectively,} we then estimate
\begin{equation}
    \begin{split}
        w_1^{(k)}(x_{j-1}) &\geq \exp(-C_1h_j)w_1^{(k)}(x_{j}) + w_1^{(k)}(x_j)(1-\exp(-C_1h_j)) = w_1^{(k)}(x_j), \\
        w^{(k)}_2(x_{j-1}) &\geq \exp(-C_2h_j)w^{(k)}_2(x_{j}) + w_2^{(k)}(x_j)(1-\exp(-C_2h_j)) = w_2^{(k)}(x_j).
    \end{split}
\end{equation}
\end{proof}


\section{Continuity and feasibility of the solutions on a star graph}\label{chapter3:network}
In \S\ref{sec:junction_values}, we compute the values at the junction obtained from the boundary data on each branch. Using these values, we impose the continuity conditions of \S\ref{sec:prelim_network} and determine the boundary values for which a continuous solution exists. Since on each branch we use the scheme described in \S\ref{sec:ch_2_1D_system} with data given at the boundary vertex, the results of \S\ref{sec:analysis} apply branch by branch.

\subsection{Values at the junction of a star graph from the boundary data}\label{sec:junction_values}
On each branch of the network, we apply the discretization of \S\ref{sec:ch_2_1D_system}, starting from the boundary values at $\bm x_m(\ell_m)$, and we obtain the values of the solutions at the junction.

\label{sec:ch_3_network_system}We now consider the problem \eqref{eq:net_system}. On each branch, the data are given only at the boundary vertex $\bm x_m(\ell_m)$, and we use only the iteration from the boundary vertex to the junction, that is, the iterates $\widetilde{\bm w}_{i, m}^{(k)}$, $\bm c^{(k)}$, and $\bm d^{(k)}$.
We adapt the notation to the problem on the star graph and denote by $\bm \beta_i^{(m)}$, $\bm \phi_i^{(m)}$, $M_{i, \pm}^{(m)}$, $Z_{i, \pm}^{(m)}$, $\bm \omega_{i, \pm}^{(m)}$, and $\bm a_1^{(m)}, \ldots, \bm d_1^{(m)}$ the quantities $\bm \beta_i$ and $\bm \phi_i$ defined in \S\ref{sec:ch_2_specific_case} and in \eqref{eq:ch_2_M}--\eqref{eq:ch_2_a0b0c0d0}, with $C_i$, $\bm h$, $N$, and $\bm F_i$ replaced by $C_{i, m}$, $\bm h^{(m)}$, $N_m$, and $\bm F_i^{(m)}$, and with the boundary values $g_i(x_0)$ and $g_i(x_N)$ replaced by $g_{i, m}(0)$ and $g_{i, m}(\ell_m)$.
With this notation, the limits of the iteration from the boundary vertex, computed as in Algorithm \ref{alg:ch_2_sys-genC-nonhom-limit}, are
\begin{equation}\label{eq:ch_3_network_w1cw2d}
    \begin{split}
        \widetilde{\bm w}_{1, m}^{(\infty)} &= M_{1, +}^{(m)} \bm d^{(m, \infty)} + \bm \omega_{1, +}^{(m)}, \\
        \bm c^{(m, \infty)} &= \left( I - Z_{1, +}^{(m)} Z_{2, +}^{(m)} \right)^{-1} \left( Z_{1, +}^{(m)} \bm d_1^{(m)} + \bm c_1^{(m)} \right), \\
        \widetilde{\bm w}_{2, m}^{(\infty)} &= M_{2, +}^{(m)} \bm c^{(m, \infty)} + \bm \omega_{2, +}^{(m)}, \\
        \bm d^{(m, \infty)} &= \left( I - Z_{2, +}^{(m)} Z_{1, +}^{(m)} \right)^{-1} \left( Z_{2, +}^{(m)} \bm c_1^{(m)} + \bm d_1^{(m)} \right).
    \end{split}
\end{equation}
We thus have
\begin{equation}
    \begin{split}
        \widetilde{w}_{1, m}(0) &= \bm e_1^T M_{1, +}^{(m)} \bm d^{(m, \infty)} + \bm e_1^T \bm \omega_{1, +}^{(m)}, \\
        \widetilde{w}_{2, m}(0) &= \bm e_1^T M_{2, +}^{(m)} \bm c^{(m, \infty)} + \bm e_1^T \bm \omega_{2, +}^{(m)}.
    \end{split}
\end{equation}
Since we know
\begin{equation}
    \bm e_1^T \left( A_{i, +}^{(m)} \right)^{-1} \bm e_{N_m} = \exp(-C_{i, m} \sum_{j=1}^{N_m-1} h_j^{(m)}),
\end{equation}
we can rewrite $\bm e_1^T \bm \omega_{i, +}^{(m)}$ as
\begin{equation}
    \bm e_1^T \bm \omega_{i, +}^{(m)} = \frac{1}{C_{i, m}} \bm e_1^T M_{i, +}^{(m)} \bm F_i^{(m)} + g_{i, m}(\ell_m) \exp(-C_{i, m} \ell_m).
\end{equation}
Substituting this into $\widetilde{w}_{i, m}(0)$, we have
\begin{equation}
    \begin{split}
        \widetilde{w}_{1, m}(0) &= \bm e_1^T M_{1, +}^{(m)} \left( \bm d^{(m, \infty)} + \frac{\bm F_1^{(m)}}{C_{1, m}} \right) + g_{1, m}(\ell_m) \exp(-C_{1, m} \ell_m), \\
        \widetilde{w}_{2, m}(0) &= \bm e_1^T M_{2, +}^{(m)} \left( \bm c^{(m, \infty)} + \frac{\bm F_2^{(m)}}{C_{2, m}} \right) + g_{2, m}(\ell_m) \exp(-C_{2, m} \ell_m).
    \end{split}
\end{equation}
We notice that
\begin{equation}
    \begin{split}
        \bm c_1^{(m)} &= \frac{1}{C_{1, m}} Z_{1, +}^{(m)} \bm F_1^{(m)} + g_{1, m}(\ell_m) \bm \psi_{1, +}^{(m)}, \\
        \bm d_1^{(m)} &= \frac{1}{C_{2, m}} Z_{2, +}^{(m)} \bm F_2^{(m)} + g_{2, m}(\ell_m) \bm \psi_{2, +}^{(m)},
    \end{split}
\end{equation}
where
\begin{equation}
    \bm \psi_{i, +}^{(m)} = \left( I - \operatorname{diag}\left(\bm \beta_i^{(m)}\right) \right) \left( I + \exp(-C_{i, m} h_{N_m}^{(m)}) P_+ \left( A_{i, +}^{(m)} \right)^{-1} \right) \bm e_{N_m}.
\end{equation}
Then we can write the values of the solutions at the junction (the origin) as
\begin{equation}\label{eq:net_system_junction_values}
    \widetilde{w}_{i, m}(0) = \mathfrak{a}_{i, +}^{(m)} + g_{i, m}(\ell_m) \mathfrak{b}_{i, +}^{(m)} + g_{3-i, m}(\ell_m) \mathfrak{c}_{i, +}^{(m)},
\end{equation}
where, for $i = 1, 2$,
\begin{equation}
    \begin{split}
        \mathfrak{a}_{i, +}^{(m)} &= \bm e_1^T M_{i, +}^{(m)} \left( I - Z_{3-i, +}^{(m)} Z_{i, +}^{(m)} \right)^{-1} \left( \frac{1}{C_{i, m}} \bm F_i^{(m)} + \frac{1}{C_{3-i, m}} Z_{3-i, +}^{(m)} \bm F_{3-i}^{(m)} \right), \\
        \mathfrak{b}_{i, +}^{(m)} &= \left( \bm e_1^T M_{i, +}^{(m)} \left( I - Z_{3-i, +}^{(m)} Z_{i, +}^{(m)} \right)^{-1} Z_{3-i, +}^{(m)} \bm \psi_{i, +}^{(m)} + \exp(-C_{i, m} \ell_m) \right), \\
        \mathfrak{c}_{i, +}^{(m)} &= \bm e_1^T M_{i, +}^{(m)} \left( I - Z_{3-i, +}^{(m)} Z_{i, +}^{(m)} \right)^{-1} \bm \psi_{3-i, +}^{(m)}.
    \end{split}
\end{equation}
That is, for given boundary data, we can determine the values of the solution at the junction (the origin) explicitly if we only propagate from the boundaries to the junction (the origin).

\subsection{Continuity at the junction}\label{sec:junction_system}\label{subsec:ch_3_network_system_continuity}
We now consider the continuity of the solutions on each branch at the junction and we focus on determining the boundary values $g_{i, m}(\ell_m)$ so that the values of the solutions at the junction match. Clearly, such continuity conditions are satisfied if
\begin{equation}
    \widetilde{w}_{i, p}(0) = \widetilde{w}_{i, q}(0), \quad i = 1, 2, \quad p, q = 1:M.
\end{equation}
By \eqref{eq:net_system_junction_values}, this is equivalent to
\begin{equation}
    g_{i, p}(\ell_{p}) \mathfrak{b}_{i, +}^{(p)} - g_{i, q}(\ell_{q}) \mathfrak{b}_{i, +}^{(q)} + g_{3-i, p}(\ell_{p}) \mathfrak{c}_{i, +}^{(p)} - g_{3-i, q}(\ell_{q}) \mathfrak{c}_{i, +}^{(q)} = \mathfrak{a}_{i, +}^{(q)} - \mathfrak{a}_{i, +}^{(p)}.
\end{equation}
It is sufficient to take $q = p + 1$ and form a system of $2M-2$ equations in the $2M$ unknowns $g_{i, m}(\ell_m)$ for $i = 1, 2$ and $m = 1, \ldots, M$. We define the following matrices and vectors:
\begin{equation}
    \mathfrak{D}_+^{(m)} = \begin{bmatrix}
        \mathfrak{b}_{1,+}^{(m)} & \mathfrak{c}_{1,+}^{(m)} \\
        {\mathfrak{c}_{2,+}^{(m)}} & {\mathfrak{b}_{2,+}^{(m)}}
    \end{bmatrix}, \quad \mathfrak{g}_+^{(m)} = \begin{bmatrix}
        g_{1,m}(\ell_m) \\
        g_{2,m}(\ell_m)
    \end{bmatrix}, \quad \text{and} \quad \mathfrak{a}_+^{(m)} = \begin{bmatrix}
        \mathfrak{a}_{1,+}^{(m)} \\
        \mathfrak{a}_{2,+}^{(m)}
    \end{bmatrix}.
\end{equation}
We can write the $m^{\text{th}}$ pair of rows of the full system as
\begin{equation}
    \mathfrak{D}_+^{(m)} \mathfrak{g}_+^{(m)} - \mathfrak{D}_+^{(m+1)} \mathfrak{g}_+^{(m+1)} = \mathfrak{a}_+^{(m+1)} - \mathfrak{a}_+^{(m)},
\end{equation}
for $m = 1, \dots, M-1$. Summing both sides of these equations from $1$ to $m-1$ gives
\begin{equation}\label{eq:ch_3_mathfrakD_telescope}
    \mathfrak{D}_+^{(m)} \mathfrak{g}_+^{(m)} = \mathfrak{D}_+^{(1)} \mathfrak{g}_+^{(1)} + \mathfrak{a}_+^{(1)} - \mathfrak{a}_+^{(m)},
\end{equation}
for $m = 1, \dots, M$.

{
\begin{lemma}\label{lem:bc_positive}
For $i = 1, 2$ and $m = 1:M$, we have $\mathfrak{b}_{i, +}^{(m)} \geq \exp(-C_{i, m} \ell_m) > 0$ and $\mathfrak{c}_{i, +}^{(m)} > 0$.
\end{lemma}
\begin{proof}
Fix $m$ and $i$, and write $k = 3 - i$. All inequalities between matrices and vectors below are entrywise. By the formula for the entries of $A_{i, +}^{-1}$ in \S\ref{sec:discretization}, the entries of $\left( A_{i, +}^{(m)} \right)^{-1}$ on and above the diagonal are positive and all other entries are zero. Since $C_{i, m} > 0$, the diagonal matrix $I - \operatorname{diag}(\exp(-C_{i, m} \bm h^{(m)}))$ has positive diagonal, so $M_{i, +}^{(m)} \geq 0$ by \eqref{eq:ch_2_M}, and
\begin{equation*}
    \left( M_{i, +}^{(m)} \right)_{1 N_m} = \left( \left( A_{i, +}^{(m)} \right)^{-1} \right)_{1 N_m} \left( 1 - \exp(-C_{i, m} h_{N_m}^{(m)}) \right) > 0.
\end{equation*}
By \eqref{eq:ch_2_diag-eig}, $\beta_{i, j}^{(m)}, 1 - \beta_{i, j}^{(m)} \in (0, 1)$, and $P_+ \geq 0$, so $Z_{i, +}^{(m)} \geq 0$ by \eqref{eq:ch_2_Z_i_pm}, and the same holds for $Z_{k, +}^{(m)}$. By the definition of $\bm \psi_{k, +}^{(m)}$, we have $\bm \psi_{k, +}^{(m)} \geq \bm 0$, and since $\bm e_{N_m}^T P_+ = \bm 0^T$, the last component of $\bm \psi_{k, +}^{(m)}$ is $1 - \beta_{k, N_m}^{(m)} > 0$. The same holds for $\bm \psi_{i, +}^{(m)}$. Since $\rho\left( Z_{k, +}^{(m)} Z_{i, +}^{(m)} \right) < 1$ by \eqref{eq:ch_2_diag-eig}, and $Z_{k, +}^{(m)}$ and $Z_{i, +}^{(m)}$ are nonnegative matrices, the Neumann series gives
\begin{equation*}
    \left( I - Z_{k, +}^{(m)} Z_{i, +}^{(m)} \right)^{-1} = I + \sum_{n = 1}^{\infty} \left( Z_{k, +}^{(m)} Z_{i, +}^{(m)} \right)^n \geq I.
\end{equation*}
In the formula for $\mathfrak{b}_{i, +}^{(m)}$, the first term is a product of nonnegative factors, so $\mathfrak{b}_{i, +}^{(m)} \geq \exp(-C_{i, m} \ell_m) > 0$. Since $\bm e_1^T M_{i, +}^{(m)} \geq \bm 0^T$ and $\bm \psi_{k, +}^{(m)} \geq \bm 0$, we obtain
\begin{equation*}
    \mathfrak{c}_{i, +}^{(m)} = \bm e_1^T M_{i, +}^{(m)} \left( I - Z_{k, +}^{(m)} Z_{i, +}^{(m)} \right)^{-1} \bm \psi_{k, +}^{(m)} \geq \bm e_1^T M_{i, +}^{(m)} \bm \psi_{k, +}^{(m)} \geq \left( M_{i, +}^{(m)} \right)_{1 N_m} \left( 1 - \beta_{k, N_m}^{(m)} \right) > 0.
\end{equation*}
\end{proof}
}

\begin{theorem}\label{thrm:app_C_mathfrakD_matrix_rank_one}
{For $i = 1, 2$, we have $\mathfrak{b}_{i, +}^{(m)} + \mathfrak{c}_{i, +}^{(m)} = 1$. Consequently, $\det \left( \mathfrak{D}_+^{(m)} \right) = \mathfrak{b}_{1, +}^{(m)} + \mathfrak{b}_{2, +}^{(m)} - 1$, and $\mathfrak{D}_+^{(m)}$ has rank one if and only if $\mathfrak{b}_{1, +}^{(m)} + \mathfrak{b}_{2, +}^{(m)} = 1$, in which case its two rows are equal.}
\end{theorem}
\begin{proof}
{By its definition, $\mathfrak{a}_{i, +}^{(m)}$ is linear in $\left( \bm F_1^{(m)}, \bm F_2^{(m)} \right)$, so it vanishes when $\bm F_1^{(m)} = \bm F_2^{(m)} = \bm 0$. Take $\bm F_1^{(m)} = \bm F_2^{(m)} = \bm 0$ and $g_{1, m}(\ell_m) = g_{2, m}(\ell_m) = g$. If all components of $\bm c^{(k-1)}$ and $\bm d^{(k-1)}$ equal $g$, then on every cell $T_j^{(m)}$ the right-hand sides in \eqref{eq:ch_2_differences_vw} and \eqref{eq:ch_2_abcd-comp-genC} give $\widetilde{w}_{i, m}^{(k)}(x_{j-1}) = \exp(-C_{i, m} h_j^{(m)}) g + g \left( 1 - \exp(-C_{i, m} h_j^{(m)}) \right) = g$ and $c_{j, k} = d_{j, k} = \beta g + (1 - \beta) g = g$, with $\beta = \beta_{1, j}^{(m)}$ and $\beta = \beta_{2, j}^{(m)}$, respectively. Hence the constant vectors are a fixed point of the iteration, and since the limits are independent of the initial values by \eqref{eq:ch_2_diag-eig}, $\widetilde{w}_{i, m}(0) = g$. By \eqref{eq:net_system_junction_values}, $g = g \left( \mathfrak{b}_{i, +}^{(m)} + \mathfrak{c}_{i, +}^{(m)} \right)$ for every $g$, which gives the first claim. Therefore
\begin{equation*}
    \mathfrak{D}_+^{(m)} = \begin{bmatrix}
        \mathfrak{b}_{1,+}^{(m)} & 1 - \mathfrak{b}_{1,+}^{(m)} \\
        1 - \mathfrak{b}_{2,+}^{(m)} & \mathfrak{b}_{2,+}^{(m)}
    \end{bmatrix}, \qquad
    \det \left( \mathfrak{D}_+^{(m)} \right) = \mathfrak{b}_{1, +}^{(m)} \mathfrak{b}_{2, +}^{(m)} - \left( 1 - \mathfrak{b}_{1, +}^{(m)} \right) \left( 1 - \mathfrak{b}_{2, +}^{(m)} \right) = \mathfrak{b}_{1, +}^{(m)} + \mathfrak{b}_{2, +}^{(m)} - 1.
\end{equation*}
If the determinant vanishes, then $1 - \mathfrak{b}_{2, +}^{(m)} = \mathfrak{b}_{1, +}^{(m)}$ and $\mathfrak{b}_{2, +}^{(m)} = 1 - \mathfrak{b}_{1, +}^{(m)}$, so the two rows are equal. Since $\mathfrak{b}_{i, +}^{(m)}, \mathfrak{c}_{i, +}^{(m)} > 0$ by Lemma \ref{lem:bc_positive}, the matrix $\mathfrak{D}_+^{(m)}$ cannot have rank zero, so it has rank one.}
\end{proof}

{
\begin{theorem}\label{thrm:det_D_product}
For $m = 1:M$,
\begin{equation}\label{eq:det_D_product}
    \det \left( \mathfrak{D}_+^{(m)} \right) = \prod_{j = 1}^{N_m} \frac{\left( 1 - \beta_{1, j}^{(m)} \right) \left( 1 - \beta_{2, j}^{(m)} \right)}{1 - \beta_{1, j}^{(m)} \beta_{2, j}^{(m)}} \left( \mathfrak{B}\left( C_{1, m} h_j^{(m)} \right) + \mathfrak{B}\left( C_{2, m} h_j^{(m)} \right) - 1 \right).
\end{equation}
In particular, $\mathfrak{D}_+^{(m)}$ is singular if and only if $\mathfrak{B}\left( C_{1, m} h_j^{(m)} \right) + \mathfrak{B}\left( C_{2, m} h_j^{(m)} \right) = 1$ for some $j$, and the sign of $\det \left( \mathfrak{D}_+^{(m)} \right)$ equals the sign of $\prod_{j = 1}^{N_m} \left( \mathfrak{B}\left( C_{1, m} h_j^{(m)} \right) + \mathfrak{B}\left( C_{2, m} h_j^{(m)} \right) - 1 \right)$.
\end{theorem}
\begin{proof}
Fix $m$ and omit the superscripts $(m)$ and $(m, \infty)$. Take $\bm F_1 = \bm F_2 = \bm 0$, so that $\mathfrak{a}_{i, +} = 0$ and, by \eqref{eq:net_system_junction_values},
\begin{equation*}
    \begin{bmatrix} \widetilde{w}_1(0) \\ \widetilde{w}_2(0) \end{bmatrix} = \mathfrak{D}_+ \begin{bmatrix} g_1(\ell) \\ g_2(\ell) \end{bmatrix}.
\end{equation*}
Fix a cell $T_j$, and let $\mathfrak{e}_i = \exp(-C_i h_j)$, $\beta_i = \beta_{i, j}$, $\gamma_i = 1 - \beta_{i, j}$, and $\delta_j = 1 - \beta_1 \beta_2 = \gamma_1 + \gamma_2 - \gamma_1 \gamma_2$. By \eqref{eq:ch_2_diag-eig}, $\beta_i, \gamma_i \in (0, 1)$, so $\delta_j > 0$. By \eqref{eq:ch_2_diag-eig}, the iterates converge, and letting $k \to \infty$ in the last two relations of \eqref{eq:ch_2_abcd-comp-genC} and of \eqref{eq:ch_2_differences_vw}, the limits satisfy, with $W_i = \widetilde{w}_i(x_j)$,
\begin{equation*}
    c_j = \beta_1 d_j + \gamma_1 W_1, \quad d_j = \beta_2 c_j + \gamma_2 W_2, \quad \widetilde{w}_1(x_{j-1}) = \mathfrak{e}_1 W_1 + (1 - \mathfrak{e}_1) d_j, \quad \widetilde{w}_2(x_{j-1}) = \mathfrak{e}_2 W_2 + (1 - \mathfrak{e}_2) c_j.
\end{equation*}
The first two equations form a linear system for $(c_j, d_j)$ with determinant $\delta_j > 0$, and
\begin{equation*}
    c_j = \frac{\gamma_1 W_1 + \beta_1 \gamma_2 W_2}{\delta_j}, \qquad d_j = \frac{\beta_2 \gamma_1 W_1 + \gamma_2 W_2}{\delta_j}.
\end{equation*}
Substituting these into the last two equations gives
\begin{equation*}
    \begin{bmatrix} \widetilde{w}_1(x_{j-1}) \\ \widetilde{w}_2(x_{j-1}) \end{bmatrix} = \mathcal{T}_j \begin{bmatrix} W_1 \\ W_2 \end{bmatrix}, \qquad
    \mathcal{T}_j = \begin{bmatrix}
        \mathfrak{e}_1 + (1 - \mathfrak{e}_1) \beta_2 \gamma_1 / \delta_j & (1 - \mathfrak{e}_1) \gamma_2 / \delta_j \\
        (1 - \mathfrak{e}_2) \gamma_1 / \delta_j & \mathfrak{e}_2 + (1 - \mathfrak{e}_2) \beta_1 \gamma_2 / \delta_j
    \end{bmatrix}.
\end{equation*}
Since $\left( \widetilde{w}_1(x_N), \widetilde{w}_2(x_N) \right) = \left( g_1(\ell), g_2(\ell) \right)$, applying this relation for $j = N, N-1, \ldots, 1$ gives, for all boundary values,
\begin{equation*}
    \mathfrak{D}_+ = \mathcal{T}_1 \mathcal{T}_2 \cdots \mathcal{T}_N, \qquad \det \left( \mathfrak{D}_+ \right) = \prod_{j = 1}^{N} \det \left( \mathcal{T}_j \right).
\end{equation*}
Each row of $\mathcal{T}_j$ sums to one, since $(1 - \mathfrak{e}_1)(\beta_2 \gamma_1 + \gamma_2) = (1 - \mathfrak{e}_1) \delta_j$ and $(1 - \mathfrak{e}_2)(\gamma_1 + \beta_1 \gamma_2) = (1 - \mathfrak{e}_2) \delta_j$. Hence, as in the proof of Theorem \ref{thrm:app_C_mathfrakD_matrix_rank_one}, $\det \left( \mathcal{T}_j \right) = \left( \mathcal{T}_j \right)_{11} + \left( \mathcal{T}_j \right)_{22} - 1$, and
\begin{equation*}
    \delta_j \det \left( \mathcal{T}_j \right) = (1 - \mathfrak{e}_1) \beta_2 \gamma_1 + (1 - \mathfrak{e}_2) \beta_1 \gamma_2 - (1 - \mathfrak{e}_1 - \mathfrak{e}_2) \delta_j = \mathfrak{e}_1 \gamma_2 + \mathfrak{e}_2 \gamma_1 - \gamma_1 \gamma_2,
\end{equation*}
where the last equality follows by expanding with $\beta_i = 1 - \gamma_i$ and $\delta_j = \gamma_1 + \gamma_2 - \gamma_1 \gamma_2$. From the definition \eqref{eq:bernoulli}, $\mathfrak{B}(-t) = e^t \mathfrak{B}(t)$, so $\mathfrak{B}(C_i h_j) = \mathfrak{e}_i \mathfrak{B}(-C_i h_j) = \mathfrak{e}_i / \gamma_i$. Therefore
\begin{equation*}
    \det \left( \mathcal{T}_j \right) = \frac{\gamma_1 \gamma_2}{\delta_j} \left( \mathfrak{B}(C_1 h_j) + \mathfrak{B}(C_2 h_j) - 1 \right),
\end{equation*}
and taking the product over $j$ gives \eqref{eq:det_D_product}. Since $\gamma_1 \gamma_2 / \delta_j = (1 - \beta_{1, j})(1 - \beta_{2, j}) / (1 - \beta_{1, j} \beta_{2, j}) > 0$ for every $j$, the prefactor in \eqref{eq:det_D_product} is positive, which gives the two remaining claims.
\end{proof}

\begin{remark}
The factorization $\mathfrak{D}_+^{(m)} = \mathcal{T}_1 \cdots \mathcal{T}_{N_m}$ in the proof of Theorem \ref{thrm:det_D_product} also gives the other properties of $\mathfrak{D}_+^{(m)}$ used above. Each $\mathcal{T}_j$ has positive entries and rows summing to one, and both properties are preserved under products. Hence $\mathfrak{b}_{i, +}^{(m)} + \mathfrak{c}_{i, +}^{(m)} = 1$ and $\mathfrak{b}_{i, +}^{(m)}, \mathfrak{c}_{i, +}^{(m)} > 0$, as in Theorem \ref{thrm:app_C_mathfrakD_matrix_rank_one} and Lemma \ref{lem:bc_positive}. By Theorem \ref{thrm:det_D_product}, whether $\mathfrak{D}_+^{(m)}$ is singular depends only on $C_{1, m}$, $C_{2, m}$, and $\bm h^{(m)}$: it is singular if and only if $\mathfrak{B}\left( C_{1, m} h_j^{(m)} \right) + \mathfrak{B}\left( C_{2, m} h_j^{(m)} \right) = 1$ for some $j$, which can be checked cell by cell without computing $\mathfrak{b}_{i, +}^{(m)}$.
\end{remark}
}

{
\begin{lemma}\label{lem:diff_row}
For $m = 1:M$, we have $\begin{bmatrix} 1 & -1 \end{bmatrix} \mathfrak{D}_+^{(m)} = \det \left( \mathfrak{D}_+^{(m)} \right) \begin{bmatrix} 1 & -1 \end{bmatrix}$. Consequently,
\begin{equation}\label{eq:diff_junction}
    \widetilde{w}_{1, m}(0) - \widetilde{w}_{2, m}(0) = \mathfrak{a}_{1, +}^{(m)} - \mathfrak{a}_{2, +}^{(m)} + \det \left( \mathfrak{D}_+^{(m)} \right) \left( g_{1, m}(\ell_m) - g_{2, m}(\ell_m) \right).
\end{equation}
\end{lemma}
\begin{proof}
By Theorem \ref{thrm:app_C_mathfrakD_matrix_rank_one}, $\mathfrak{c}_{i, +}^{(m)} = 1 - \mathfrak{b}_{i, +}^{(m)}$ for $i = 1, 2$, and therefore
\begin{equation*}
    \begin{bmatrix} 1 & -1 \end{bmatrix} \mathfrak{D}_+^{(m)} = \begin{bmatrix} \mathfrak{b}_{1, +}^{(m)} - 1 + \mathfrak{b}_{2, +}^{(m)} & 1 - \mathfrak{b}_{1, +}^{(m)} - \mathfrak{b}_{2, +}^{(m)} \end{bmatrix} = \left( \mathfrak{b}_{1, +}^{(m)} + \mathfrak{b}_{2, +}^{(m)} - 1 \right) \begin{bmatrix} 1 & -1 \end{bmatrix},
\end{equation*}
and $\mathfrak{b}_{1, +}^{(m)} + \mathfrak{b}_{2, +}^{(m)} - 1 = \det \left( \mathfrak{D}_+^{(m)} \right)$ by the same theorem. By \eqref{eq:net_system_junction_values}, the vector of the values $\widetilde{w}_{1, m}(0)$ and $\widetilde{w}_{2, m}(0)$ equals $\mathfrak{a}_+^{(m)} + \mathfrak{D}_+^{(m)} \mathfrak{g}_+^{(m)}$, and multiplying by $\begin{bmatrix} 1 & -1 \end{bmatrix}$ gives \eqref{eq:diff_junction}.
\end{proof}
}

{By Lemma \ref{lem:bc_positive}, $\mathfrak{D}_+^{(m)}$ is not rank zero, and we consider two cases: $\mathfrak{D}_+^{(m)}$ has rank one, and $\mathfrak{D}_+^{(m)}$ is nonsingular.}

Suppose $\mathfrak{D}_+^{(m)}$ has rank one. By Theorem \ref{thrm:app_C_mathfrakD_matrix_rank_one}, this occurs when {$\mathfrak{b}_{1, +}^{(m)} + \mathfrak{b}_{2, +}^{(m)} = 1$, and then the two rows are equal}, so {the first row of equation \eqref{eq:ch_3_mathfrakD_telescope} determines $\mathfrak{g}_+^{(m)}$, and the second row gives one additional condition, which we state below}. Since $\mathfrak{c}_{i, +}^{(m)} > 0$ {by Lemma \ref{lem:bc_positive}}, we are able to solve for $g_{2, m}(\ell_m)$ explicitly:
\begin{equation}
    g_{2, m}(\ell_m) = \frac{\mathfrak{b}_{1, +}^{(1)} g_{1, 1}(\ell_1) + \mathfrak{c}_{1, +}^{(1)} g_{2, 1}(\ell_1) - \mathfrak{b}_{1, +}^{(m)} g_{1, m}(\ell_m) + \mathfrak{a}_{1, +}^{(1)} - \mathfrak{a}_{1, +}^{(m)}}{\mathfrak{c}_{1, +}^{(m)}}.
\end{equation}
That is, if we know $g_{1, 1}(\ell_1)$, $g_{2, 1}(\ell_1)$, and $g_{1, m}(\ell_m)$, then we can determine $g_{2, m}(\ell_m)$.

{
The first row alone is not sufficient. Multiplying \eqref{eq:ch_3_mathfrakD_telescope} by $\begin{bmatrix} 1 & -1 \end{bmatrix}$ and applying Lemma \ref{lem:diff_row} to $\mathfrak{D}_+^{(m)}$ and $\mathfrak{D}_+^{(1)}$, we obtain, for every $m$,
\begin{equation}\label{eq:diff_telescope}
    \det \left( \mathfrak{D}_+^{(m)} \right) \left( g_{1, m}(\ell_m) - g_{2, m}(\ell_m) \right) = \det \left( \mathfrak{D}_+^{(1)} \right) \left( g_{1, 1}(\ell_1) - g_{2, 1}(\ell_1) \right) + \mathfrak{a}_{1, +}^{(1)} - \mathfrak{a}_{2, +}^{(1)} - \mathfrak{a}_{1, +}^{(m)} + \mathfrak{a}_{2, +}^{(m)}.
\end{equation}
The second row of \eqref{eq:ch_3_mathfrakD_telescope} is the first row minus \eqref{eq:diff_telescope}. When $\det \left( \mathfrak{D}_+^{(m)} \right) = 0$, the left-hand side of \eqref{eq:diff_telescope} vanishes, and, given the first row, the second row holds if and only if
\begin{equation}\label{eq:rank_one_compat}
    \det \left( \mathfrak{D}_+^{(1)} \right) \left( g_{1, 1}(\ell_1) - g_{2, 1}(\ell_1) \right) = \left( \mathfrak{a}_{1, +}^{(m)} - \mathfrak{a}_{2, +}^{(m)} \right) - \left( \mathfrak{a}_{1, +}^{(1)} - \mathfrak{a}_{2, +}^{(1)} \right).
\end{equation}
If $\mathfrak{D}_+^{(1)}$ is nonsingular, then \eqref{eq:rank_one_compat} determines $g_{1, 1}(\ell_1) - g_{2, 1}(\ell_1)$. If $\mathfrak{D}_+^{(1)}$ is also singular, then \eqref{eq:rank_one_compat} involves only the data $\mathfrak{a}_+^{(1)}$ and $\mathfrak{a}_+^{(m)}$, and either it holds for all boundary values or no boundary values give solutions that are continuous at the junction. In terms of the values at the junction, by \eqref{eq:diff_junction} a branch with singular $\mathfrak{D}_+^{(m)}$ gives $\widetilde{w}_{1, m}(0) - \widetilde{w}_{2, m}(0) = \mathfrak{a}_{1, +}^{(m)} - \mathfrak{a}_{2, +}^{(m)}$ for all boundary values on that branch, and continuity requires the same difference on every other branch.
}

In what follows, we disregard the case when $\mathfrak{D}_+^{(m)}$ is singular, that is, we assume that the matrices $\mathfrak{D}_+^{(m)}$ are nonsingular; by Theorem \ref{thrm:app_C_mathfrakD_matrix_rank_one}, this {means $\mathfrak{b}_{1, +}^{(m)} + \mathfrak{b}_{2, +}^{(m)} \neq 1$}. We define
\begin{equation}\label{eq:rev_T_xi}
    T^{(m)} = \left( \mathfrak{D}_+^{(m)} \right)^{-1} \mathfrak{D}_+^{(1)}
    \quad \text{and} \quad
    \bm \xi^{(m)} = \left( \mathfrak{D}_+^{(m)} \right)^{-1} \left( \mathfrak{a}_+^{(1)} - \mathfrak{a}_+^{(m)} \right),
\end{equation}
so that $T^{(1)} = I$ and $\bm \xi^{(1)} = \bm 0$.

\begin{lemma}\label{lem:rev_continuity}
The boundary values $\mathfrak{g}_+^{(m)}$, $m = 1:M$, give solutions that are continuous at the junction if and only if
\begin{equation}\label{eq:rev_gm}
    \mathfrak{g}_+^{(m)} = T^{(m)} \mathfrak{g}_+^{(1)} + \bm \xi^{(m)}, \quad m = 1:M.
\end{equation}
\end{lemma}
\begin{proof}
Multiply \eqref{eq:ch_3_mathfrakD_telescope} by $\left( \mathfrak{D}_+^{(m)} \right)^{-1}$.
\end{proof}

\subsection{Feasibility on one branch}
There are additional criteria that the boundary conditions must satisfy. Taking the limit as $k \to \infty$ of the conditions in Lemma \ref{lem:ch_2_1D_2_eq_monotone_scheme}, and altering the notation to align with that for the network, we have the following:
\begin{equation}
    \begin{split}
        \bm F_1^{(m)} + C_{1, m} \bm d^{(m, \infty)} &\geq C_{1, m} \widetilde{\bm w}_{1, m}^{(\infty)}, \\
        \bm F_2^{(m)} + C_{2, m} \bm c^{(m, \infty)} &\geq C_{2, m} \widetilde{\bm w}_{2, m}^{(\infty)}.
    \end{split}
\end{equation}
Here the $j$th component of $\widetilde{\bm w}_{i, m}^{(\infty)}$ is the value at the left end $x_{j-1}$ of $T_j^{(m)}$, while Lemma \ref{lem:ch_2_1D_2_eq_monotone_scheme} uses the value at the right end $x_j$. The two sets of conditions are equivalent: with $K = F_{1, j}^{(m)} / C_{1, m} + d_j^{(m, \infty)}$, the relation \eqref{eq:ch_2_differences_vw} gives $K - \widetilde{w}_{1, m}(x_{j-1}) = \exp(-C_{1, m} h_j^{(m)}) \left( K - \widetilde{w}_{1, m}(x_j) \right)$, and similarly for $\widetilde{w}_{2, m}$.
Substituting equation \eqref{eq:ch_3_network_w1cw2d} into the above inequalities, we get
\begin{equation}
    \begin{split}
        \bm F_1^{(m)} + C_{1, m} \bm d^{(m, \infty)} &\geq C_{1, m} \left( M_{1, +}^{(m)} \bm d^{(m, \infty)} + \bm \omega_{1, +}^{(m)} \right), \\
        \bm F_2^{(m)} + C_{2, m} \bm c^{(m, \infty)} &\geq C_{2, m} \left( M_{2, +}^{(m)} \bm c^{(m, \infty)} + \bm \omega_{2, +}^{(m)} \right),
    \end{split}
\end{equation}
which{, after dividing by $C_{i, m} > 0$,} we can rewrite as
\begin{equation}
    \begin{split}
        \left( I - M_{1, +}^{(m)} \right) \frac{\bm F_1^{(m)}}{C_{1, m}} &\geq -\left( I - M_{1, +}^{(m)} \right) \bm d^{(m, \infty)} + g_{1, m}(\ell_m) \exp(-C_{1, m} h_{N_m}^{(m)}) \left( A_{1, +}^{(m)} \right)^{-1} \bm e_{N_m}, \\
        \left( I - M_{2, +}^{(m)} \right) \frac{\bm F_2^{(m)}}{C_{2, m}} &\geq -\left( I - M_{2, +}^{(m)} \right) \bm c^{(m, \infty)} + g_{2, m}(\ell_m) \exp(-C_{2, m} h_{N_m}^{(m)}) \left( A_{2, +}^{(m)} \right)^{-1} \bm e_{N_m}.
    \end{split}
\end{equation}
We further rewrite this as
\begin{equation}
    \begin{split}
        0 &\leq \mathfrak{q}_{1, j}^{(m)} + g_{1, m}(\ell_m) \mathfrak{r}_{1, j}^{(m)} + g_{2, m}(\ell_m) \mathfrak{s}_{1, j}^{(m)}, \\
        0 &\leq \mathfrak{q}_{2, j}^{(m)} + g_{1, m}(\ell_m) \mathfrak{s}_{2, j}^{(m)} + g_{2, m}(\ell_m) \mathfrak{r}_{2, j}^{(m)},
    \end{split}
\end{equation}
for $j = 1:N_m$, where
\begin{equation}
    \begin{split}
        \mathfrak{q}_{i, j}^{(m)} &= \bm e_j^T \left( I - M_{i, +}^{(m)} \right) \left( I - Z_{3-i, +}^{(m)} Z_{i, +}^{(m)} \right)^{-1} \left( \frac{1}{C_{i, m}} \bm F_i^{(m)} + \frac{1}{C_{3-i, m}} Z_{3-i, +}^{(m)} \bm F_{3-i}^{(m)} \right), \\
        \mathfrak{r}_{i, j}^{(m)} &= \bm e_j^T \left( I - M_{i, +}^{(m)} \right) \left( I - Z_{3-i, +}^{(m)} Z_{i, +}^{(m)} \right)^{-1} Z_{3-i, +}^{(m)} \bm \psi_{i, +}^{(m)} \\
        &\qquad\qquad - \exp(-C_{i, m} h_{N_m}^{(m)}) \, \bm e_j^T \left( A_{i, +}^{(m)} \right)^{-1} \bm e_{N_m}, \\
        \mathfrak{s}_{i, j}^{(m)} &= \bm e_j^T \left( I - M_{i, +}^{(m)} \right) \left( I - Z_{3-i, +}^{(m)} Z_{i, +}^{(m)} \right)^{-1} \bm \psi_{3-i, +}^{(m)}, \quad i = 1, 2.
    \end{split}
\end{equation}
The quantities $\mathfrak{q}_{i, j}^{(m)}$, $\mathfrak{r}_{i, j}^{(m)}$, and $\mathfrak{s}_{i, j}^{(m)}$ do not depend on the boundary values. Together with the conditions $g_{i, m}(\ell_m) \geq 0$, these are $2 N_m + 2$ linear inequalities for $\mathfrak{g}_+^{(m)}$. We denote by
\begin{equation}\label{eq:rev_Pm}
    P^{(m)} = \left\{ \mathfrak{g} = \begin{bmatrix} g_1 \\ g_2 \end{bmatrix} \in \mathbb{R}^2 : g_1, g_2 \geq 0, \ \
    \begin{aligned}
        & \mathfrak{q}_{1, j}^{(m)} + g_1 \mathfrak{r}_{1, j}^{(m)} + g_2 \mathfrak{s}_{1, j}^{(m)} \geq 0, \\
        & \mathfrak{q}_{2, j}^{(m)} + g_1 \mathfrak{s}_{2, j}^{(m)} + g_2 \mathfrak{r}_{2, j}^{(m)} \geq 0, \quad j = 1:N_m
    \end{aligned}
    \right\},
\end{equation}
the set of feasible boundary values on the branch $J_m$. As an intersection of half-planes, $P^{(m)}$ is a convex polygon, possibly empty or unbounded.

The coefficients of the second family of inequalities are proportional to those of the first.
\begin{lemma}\label{lem:proportional}
For $j = 1:N_m$ and $m = 1:M$, let
\begin{equation}\label{eq:lambda_j}
    \lambda_j^{(m)} = -\frac{C_{2, m} \left( \exp(C_{1, m} h_j^{(m)}) - 1 \right)}{C_{1, m} \left( \exp(C_{2,m} h_j^{(m)}) - 1 \right)} < 0.
\end{equation}
Then $\mathfrak{s}_{2, j}^{(m)} = \lambda_j^{(m)} \mathfrak{r}_{1, j}^{(m)}$ and $\mathfrak{r}_{2, j}^{(m)} = \lambda_j^{(m)} \mathfrak{s}_{1, j}^{(m)}$. In particular, $\Delta_j^{(m)} =\mathfrak{r}_{1, j}^{(m)} \mathfrak{r}_{2, j}^{(m)} - \mathfrak{s}_{1, j}^{(m)} \mathfrak{s}_{2, j}^{(m)} = 0$.
\end{lemma}
\begin{proof}
Fix $m$ and $j$, and omit the superscripts $(m)$ and $(m, \infty)$. Let ${\mathfrak{e}_i} = \exp(-C_i h_j)$ and $\gamma_i = 1 - \beta_{i, j} = (1 - {\mathfrak{e}_i})/(C_i h_j)$ for $i = 1, 2$, and let $c_j$ and $d_j$ be the components of the limits $\bm c$ and $\bm d$. Set
\begin{equation*}
    K_1 = \frac{F_{1, j}}{C_1} + d_j, \qquad K_2 = \frac{F_{2, j}}{C_2} + c_j, \qquad X = K_1 - \widetilde{w}_1(x_j), \qquad Y = K_2 - \widetilde{w}_2(x_j).
\end{equation*}
By the derivation of \eqref{eq:rev_Pm}, the $j$th components of the two inequalities there are $K_1 - \widetilde{w}_1(x_{j-1}) \geq 0$ and $K_2 - \widetilde{w}_2(x_{j-1}) \geq 0$, and by the equivalence above,
\begin{equation*}
    \mathfrak{q}_{1, j} + g_1 \mathfrak{r}_{1, j} + g_2 \mathfrak{s}_{1, j} = {\mathfrak{e}_1} X, \qquad
    \mathfrak{q}_{2, j} + g_1 \mathfrak{s}_{2, j} + g_2 \mathfrak{r}_{2, j} = {\mathfrak{e}_2} Y,
\end{equation*}
for all boundary values $g_1, g_2$. In the limit, \eqref{eq:ch_2_abcd-comp-genC} gives $c_j = \beta_{1, j} K_1 + \gamma_1 \widetilde{w}_1(x_j)$ and $d_j = \beta_{2, j} K_2 + \gamma_2 \widetilde{w}_2(x_j)$. Hence
\begin{equation*}
    Y = \frac{F_{2, j}}{C_2} + \widetilde{w}_1(x_j) - \widetilde{w}_2(x_j) + \beta_{1, j} X, \qquad
    X = \frac{F_{1, j}}{C_1} + \widetilde{w}_2(x_j) - \widetilde{w}_1(x_j) + \beta_{2, j} Y,
\end{equation*}
and adding the two relations gives
\begin{equation*}
    \frac{\gamma_1}{{\mathfrak{e}_1}} \left( \mathfrak{q}_{1, j} + g_1 \mathfrak{r}_{1, j} + g_2 \mathfrak{s}_{1, j} \right) + \frac{\gamma_2}{{\mathfrak{e}_2}} \left( \mathfrak{q}_{2, j} + g_1 \mathfrak{s}_{2, j} + g_2 \mathfrak{r}_{2, j} \right) = \gamma_1 X + \gamma_2 Y = \frac{F_{1, j}}{C_1} + \frac{F_{2, j}}{C_2}.
\end{equation*}
The right-hand side does not depend on $g_1$ and $g_2$. Comparing the coefficients of $g_1$ and of $g_2$, we obtain $\mathfrak{s}_{2, j} = \lambda_j \mathfrak{r}_{1, j}$ and $\mathfrak{r}_{2, j} = \lambda_j \mathfrak{s}_{1, j}$ with
\begin{equation*}
    \lambda_j = -\frac{{\mathfrak{e}_2} \gamma_1}{{\mathfrak{e}_1} \gamma_2} = -\frac{C_2 \, {\mathfrak{e}_2} (1 - {\mathfrak{e}_1})}{C_1 \, {\mathfrak{e}_1} (1 - {\mathfrak{e}_2})} =-\frac{C_2 \left( \exp(C_1 h_j) - 1 \right)}{C_1 \left( \exp(C_2 h_j) - 1 \right)}.
\end{equation*}
{Since $C_1, C_2 > 0$,} all factors in the last fraction are positive, so $\lambda_j < 0$. Finally, $\Delta_j =\mathfrak{r}_{1, j} \lambda_j \mathfrak{s}_{1, j} - \mathfrak{s}_{1, j} \lambda_j \mathfrak{r}_{1, j} = 0$.
\end{proof}
Define
\begin{equation}
    \mu_j^{(m)} = g_{1, m}(\ell_m) \mathfrak{r}_{1, j}^{(m)} + g_{2, m}(\ell_m) \mathfrak{s}_{1, j}^{(m)},
\end{equation}
so that, since $\lambda_j^{(m)} < 0$, the system of inequalities becomes
\begin{equation}    
    -\mathfrak{q}_{1, j}^{(m)} \leq \mu_j^{(m)} \leq \frac{\mathfrak{q}_{2, j}^{(m)}}{|\lambda_j^{(m)}|},
\end{equation}
for $j = 1:N_m$ and $m = 1:M$. Let
\begin{equation}
    S_j^{(m)} = \left\{ \begin{bmatrix}
        g_{1, m}(\ell_m) \\
        g_{2, m}(\ell_m)
    \end{bmatrix} : -\mathfrak{q}_{1, j}^{(m)} \leq \mu_j^{(m)} \leq \frac{\mathfrak{q}_{2, j}^{(m)}}{|\lambda_j^{(m)}|} \right\},
\end{equation}
denote the feasible set, which is the region in between two parallel lines. Intersecting over $j = 1:N_m$, the feasible set $\cap_{j = 1:N_m} S_j^{(m)}$ is a convex polygon. {By Lemma \ref{lem:rev_strip} below, all these lines are parallel to $(1, 1)^T$, and the intersection is a strip.}

{
\begin{lemma}\label{lem:rev_strip}
For $j = 1:N_m$ and $i = 1, 2$, we have $\mathfrak{r}_{i, j}^{(m)} + \mathfrak{s}_{i, j}^{(m)} = 0$, and hence $\mu_j^{(m)} = \mathfrak{r}_{1, j}^{(m)} \left( g_{1, m}(\ell_m) - g_{2, m}(\ell_m) \right)$ and
\begin{equation}\label{eq:rev_strip}
    P^{(m)} = \left\{ \mathfrak{g} \in \mathbb{R}^2 : g_1, g_2 \geq 0, \ -\mathfrak{q}_{1, j}^{(m)} \leq \mathfrak{r}_{1, j}^{(m)} (g_1 - g_2) \leq \frac{\mathfrak{q}_{2, j}^{(m)}}{|\lambda_j^{(m)}|}, \ j = 1:N_m \right\}.
\end{equation}
Moreover, $-\mathfrak{q}_{1, j}^{(m)} \leq \mathfrak{q}_{2, j}^{(m)} / |\lambda_j^{(m)}|$ for $j = 1:N_m$.
\end{lemma}
\begin{proof}
Fix $m$ and $j$, and omit the superscripts $(m)$. Take $\bm F_1 = \bm F_2 = \bm 0$ and $g_1 = g_2 = g$. Then $\mathfrak{q}_{i, j} = 0$, since $\mathfrak{q}_{i, j}$ is linear in $(\bm F_1, \bm F_2)$. As in the proof of Theorem \ref{thrm:app_C_mathfrakD_matrix_rank_one}, the limits are the constant vectors with all components equal to $g$, so $K_1 = d_j = g$ and $\widetilde{w}_1(x_{j-1}) = g$ in the notation of the proof of Lemma \ref{lem:proportional}. By that proof, $g \left( \mathfrak{r}_{1, j} + \mathfrak{s}_{1, j} \right) = K_1 - \widetilde{w}_1(x_{j-1}) = 0$ for every $g$, so $\mathfrak{r}_{1, j} + \mathfrak{s}_{1, j} = 0$. By Lemma \ref{lem:proportional}, $\mathfrak{r}_{2, j} + \mathfrak{s}_{2, j} = \lambda_j \left( \mathfrak{s}_{1, j} + \mathfrak{r}_{1, j} \right) = 0$.

The quantities $\mathfrak{q}_{i, j}$, $\mathfrak{r}_{i, j}$, and $\mathfrak{s}_{i, j}$ do not depend on the boundary values, so we return to general $\bm F_1$, $\bm F_2$. By Lemma \ref{lem:proportional}, the second inequality in \eqref{eq:rev_Pm} is $\mathfrak{q}_{2, j} + \lambda_j \mu_j \geq 0$, and since $\lambda_j < 0$, the conditions of \eqref{eq:rev_Pm} for the index $j$ are $-\mathfrak{q}_{1, j} \leq \mathfrak{r}_{1, j} (g_1 - g_2) \leq \mathfrak{q}_{2, j} / |\lambda_j|$, which gives \eqref{eq:rev_strip}. For the last claim, comparing the terms independent of $g_1$ and $g_2$ in the proof of Lemma \ref{lem:proportional}, with $\mathfrak{e}_i$ and $\gamma_i$ as there,
\begin{equation*}
    \frac{\gamma_1}{\mathfrak{e}_1} \mathfrak{q}_{1, j} + \frac{\gamma_2}{\mathfrak{e}_2} \mathfrak{q}_{2, j} = \frac{F_{1, j}}{C_1} + \frac{F_{2, j}}{C_2}, \qquad \text{so} \qquad
    \mathfrak{q}_{1, j} + \frac{\mathfrak{q}_{2, j}}{|\lambda_j|} = \frac{\mathfrak{e}_1}{\gamma_1} \left( \frac{F_{1, j}}{C_1} + \frac{F_{2, j}}{C_2} \right) \geq 0,
\end{equation*}
where we used $|\lambda_j| = \mathfrak{e}_2 \gamma_1 / (\mathfrak{e}_1 \gamma_2)$ and $F_{i, j} \geq 0$.
\end{proof}
}

\subsection{Admissible boundary data on the star graph}\label{sec:junction_system_rev}
We define
\begin{equation}\label{eq:rev_G}
    G = \left\{ \mathfrak{g} \in \mathbb{R}^2 : T^{(m)} \mathfrak{g} + \bm \xi^{(m)} \in P^{(m)}, \ m = 1:M \right\}.
\end{equation}
Since $\mathfrak{g} \mapsto T^{(m)} \mathfrak{g} + \bm \xi^{(m)}$ is affine and $P^{(m)}$ is an intersection of half-planes, each condition in \eqref{eq:rev_G} is a set of $2 N_m + 2$ linear inequalities for $\mathfrak{g}$. Hence $G$ is a convex polygon, possibly empty or unbounded, defined by $\sum_{m = 1}^M (2 N_m + 2)$ linear inequalities for $\mathfrak{g}$. {By Lemma \ref{lem:rev_strip}, apart from the positivity conditions, these inequalities involve $T^{(m)} \mathfrak{g} + \bm \xi^{(m)}$ only through the difference of its two components. By Lemma \ref{lem:diff_row}, $\begin{bmatrix} 1 & -1 \end{bmatrix} \left( \mathfrak{D}_+^{(m)} \right)^{-1} = \det \left( \mathfrak{D}_+^{(m)} \right)^{-1} \begin{bmatrix} 1 & -1 \end{bmatrix}$, and therefore
\begin{equation*}
    \begin{bmatrix} 1 & -1 \end{bmatrix} \left( T^{(m)} \mathfrak{g} + \bm \xi^{(m)} \right) = \kappa^{(m)} (g_1 - g_2) + \begin{bmatrix} 1 & -1 \end{bmatrix} \bm \xi^{(m)}, \qquad \kappa^{(m)} = \frac{\det \left( \mathfrak{D}_+^{(1)} \right)}{\det \left( \mathfrak{D}_+^{(m)} \right)}.
\end{equation*}
Hence
\begin{equation}\label{eq:rev_G_strip}
    G = \left\{ \mathfrak{g} \in \mathbb{R}^2 : \
    \begin{aligned}
        & T^{(m)} \mathfrak{g} + \bm \xi^{(m)} \geq \bm 0, \\
        & -\mathfrak{q}_{1, j}^{(m)} \leq \mathfrak{r}_{1, j}^{(m)} \left( \kappa^{(m)} (g_1 - g_2) + \xi_1^{(m)} - \xi_2^{(m)} \right) \leq \frac{\mathfrak{q}_{2, j}^{(m)}}{|\lambda_j^{(m)}|}, \ j = 1:N_m, \ m = 1:M
    \end{aligned}
    \right\}.
\end{equation}
Each condition in the second line states that $g_1 - g_2$ lies in an interval, and the intersection of these intervals is an interval. Therefore $G$ is the intersection of a strip parallel to $(1, 1)^T$ with the $2M$ half-planes $T^{(m)} \mathfrak{g} + \bm \xi^{(m)} \geq \bm 0$.}

\begin{theorem}\label{thm:rev_admissible}
Suppose that $\mathfrak{D}_+^{(m)}$ is nonsingular for $m = 1:M$. The boundary values $\mathfrak{g}_+^{(m)}$, $m = 1:M$, give solutions that are continuous at the junction and satisfy the feasibility conditions on every branch if and only if $\mathfrak{g}_+^{(m)}$ is given by \eqref{eq:rev_gm} and $\mathfrak{g}_+^{(1)} \in G$.
\end{theorem}
\begin{proof}
By Lemma \ref{lem:rev_continuity}, continuity is equivalent to \eqref{eq:rev_gm}. Feasibility on the branch $J_m$ is the condition $\mathfrak{g}_+^{(m)} \in P^{(m)}$. Substituting \eqref{eq:rev_gm} gives the condition $\mathfrak{g}_+^{(1)} \in G$.
\end{proof}

Deciding whether $G$ is empty, and finding a point in it, is a linear feasibility problem in two variables, which can be solved by linear programming or by intersecting the half-planes directly. {By \eqref{eq:rev_G_strip}, it reduces to intersecting the $\sum_{m = 1}^M N_m$ intervals for $g_1 - g_2$ and the $2M$ half-planes $T^{(m)} \mathfrak{g} + \bm \xi^{(m)} \geq \bm 0$.} We summarize the computation in Algorithm \ref{alg:network_system_rev}.

\begin{algorithm}[H]
    \caption{Admissible boundary values for a system on a star graph}\label{alg:network_system_rev}
    \begin{algorithmic}
    \Function{SolveNetworkSystem}{}
    \For{$m = 1:M$}
        \State Compute $\mathfrak{a}_{i, +}^{(m)}$, $\mathfrak{b}_{i, +}^{(m)}$, $\mathfrak{c}_{i, +}^{(m)}$, $\mathfrak{D}_+^{(m)}$, and $\mathfrak{q}_{i, j}^{(m)}$, $\mathfrak{r}_{i, j}^{(m)}$, $\mathfrak{s}_{i, j}^{(m)}$ for $j = 1:N_m$;
        \State Compute $T^{(m)}$ and $\bm \xi^{(m)}$ by \eqref{eq:rev_T_xi};
        \State Form the $2 N_m + 2$ inequalities of $T^{(m)} \mathfrak{g} + \bm \xi^{(m)} \in P^{(m)}$ for $\mathfrak{g}$;
    \EndFor
    \State Find $\mathfrak{g}_+^{(1)} \in G$ by solving the linear feasibility problem \eqref{eq:rev_G}, or report that $G$ is empty;
    \State $\mathfrak{g}_+^{(m)} = T^{(m)} \mathfrak{g}_+^{(1)} + \bm \xi^{(m)}$ for $m = 2:M$;
    \State Compute $\widetilde{\bm w}_{i, m}^{(\infty)}$ for $m = 1:M$ by \eqref{eq:ch_3_network_w1cw2d};
    \EndFunction
    \end{algorithmic}
\end{algorithm}

\section{Numerical experiments}\label{sec:experiments}
We present three experiments: a manufactured solution of the system \eqref{eq:ch_2_1D_gen_system} on an interval, a star graph with three branches and a manufactured continuous solution, and a star graph with a singular branch. We examine the accuracy of the discrete solutions and illustrate the junction continuity, admissibility, and singularity results of \S\ref{chapter3:network}. 

\subsection{A manufactured solution on an interval}\label{sec:exp1}

On $(-1, 1)$ we take $C_1 = 0.5$, $C_2 = 0.8$, and
\begin{equation}\label{eq:exp1_u}
    u_1(x) = 1.5 - 0.8 |x - 0.3| - 0.3 (x - 0.3)^2, \qquad
    u_2(x) = 1.2 - 0.6 |x + 0.4| - 0.2 (x + 0.4)^2,
\end{equation}
with $f_i = |u_i'| + C_i (u_i - u_{3-i})$ and boundary values $g_i = u_i$ at $x = \pm 1$. Here $|u_1'(x)| = 0.8 + 0.6 |x - 0.3|$ and $|u_2'(x)| = 0.6 + 0.4 |x + 0.4|$ for $x \neq 0.3$ and $x \neq -0.4$, respectively, so $f_1$ and $f_2$ extend to continuous functions on $[-1, 1]$ and are positive, with $\min f_1 = 1.0165$ and $\min f_2 = 0.2256$.

The equations hold classically away from the concave kinks and in the viscosity sense at them, so \eqref{eq:exp1_u} provides a manufactured viscosity solution. Indeed, away from its kink $\xi_i$ ($\xi_1 = 0.3$, $\xi_2 = -0.4$), $u_i$ is smooth and satisfies the $i$th equation classically. At $\xi_i$ the one-sided derivatives are $\pm a_i$ with $a_1 = 0.8$ and $a_2 = 0.6$, so the kink is concave: the subdifferential of $u_i$ at $\xi_i$ is empty and the superdifferential is $[-a_i, a_i]$. For $|p| \leq a_i$ we have $|p| + C_i (u_i - u_{3-i})(\xi_i) - f_i(\xi_i) = |p| - a_i \leq 0$, so the subsolution condition holds at $\xi_i$, and the supersolution condition is void there. Hence $(u_1, u_2)$ is a viscosity solution of the system. Its uniqueness, and hence its identification with the limit of the scheme, relies on the comparison principle for weakly coupled systems with $C_1, C_2 > 0$, which we have not checked in this setting.

We use exact cell averages $F_{i, j}$ (three-point Gauss rule on each cell, split at the kinks) on uniform meshes and on random meshes with cell lengths $h_j$ proportional to independent samples of the uniform distribution on $[0.5, 1.5]$. Algorithm~\ref{alg:ch_2_sys_general} starts from the linear interpolants of the boundary values and stops when $\max_i \| \bm u_i^{(k)} - \bm u_i^{(k-1)} \|_\infty \leq 10^{-12}$. Errors are measured in the maximum norm over the nodes, that is, the error is the maximum over the nodes of $|u_i^{h}(x_j) - u_i(x_j)|$.

The stopping criterion is reached after $12$ iterations on every mesh. Table~\ref{tab:exp1} and Figure~\ref{fig:exp1} show second-order nodal accuracy, with fluctuations in the refinement rates as the kink positions vary relative to the nodes. On uniform meshes the observed rates alternate between about $1.9$ and $2.1$ for $u_1$ and between about $1.7$ and $2.3$ for $u_2$; over two refinements the rate is $2.0$. On random meshes the observed rates range from $1.0$ to $3.7$, and the average rates from $N = 16$ to $N = 2048$ are $1.9$ for $u_1$ and $2.1$ for $u_2$. Replacing the minimum $\min\{[g_i(x_0); \bm v_i], [\bm w_i; g_i(x_N)]\}$ by the fast marching method of Remark~\ref{rmk:fmm} gives the same iterates up to rounding (Table~\ref{tab:exp1_fmm}). These observations concern Algorithm~\ref{alg:ch_2_sys_general}, whose convergence we do not prove; the convergence result in \S\ref{sec:ch_2_specific_case} applies to the one-sided iterations.

\begin{table}[ht]
    \centering
    \begin{tabular}{rrrrrrr}
        \hline
        mesh & $N$ & iterations & error $u_1$ & rate & error $u_2$ & rate \\
        \hline
        uniform & 16 & 12 & 1.64e-03 & -- & 2.99e-03 & -- \\
         & 32 & 12 & 4.35e-04 & 1.92 & 5.94e-04 & 2.33 \\
         & 64 & 12 & 1.03e-04 & 2.08 & 1.88e-04 & 1.66 \\
         & 128 & 12 & 2.74e-05 & 1.91 & 3.71e-05 & 2.34 \\
         & 256 & 12 & 6.49e-06 & 2.08 & 1.18e-05 & 1.66 \\
         & 512 & 12 & 1.71e-06 & 1.92 & 2.32e-06 & 2.34 \\
         & 1024 & 12 & 4.06e-07 & 2.08 & 7.35e-07 & 1.66 \\
         & 2048 & 12 & 1.07e-07 & 1.92 & 1.45e-07 & 2.34 \\
        random & 16 & 12 & 1.40e-03 & -- & 3.34e-03 & -- \\
         & 32 & 12 & 7.18e-04 & 0.96 & 9.09e-04 & 1.88 \\
         & 64 & 12 & 1.63e-04 & 2.14 & 3.45e-04 & 1.40 \\
         & 128 & 12 & 2.14e-05 & 2.93 & 2.74e-05 & 3.66 \\
         & 256 & 12 & 1.02e-05 & 1.07 & 1.29e-05 & 1.09 \\
         & 512 & 12 & 2.55e-06 & 2.00 & 3.22e-06 & 2.00 \\
         & 1024 & 12 & 3.60e-07 & 2.82 & 6.53e-07 & 2.30 \\
         & 2048 & 12 & 1.29e-07 & 1.48 & 1.53e-07 & 2.10 \\
        \hline
    \end{tabular}
    \caption{Experiment 1: iterations of Algorithm~\ref{alg:ch_2_sys_general} and maximum nodal errors.}
    \label{tab:exp1}
\end{table}

\begin{table}[ht]
    \centering
    \begin{tabular}{rrrr}
        \hline
        $N$ & iterations, Algorithm~\ref{alg:ch_2_sys_general} & iterations, fast marching & max difference \\
        \hline
        16 & 12 & 12 & 2.2e-16 \\
        32 & 12 & 12 & 8.9e-16 \\
        64 & 12 & 12 & 7.8e-16 \\
        128 & 12 & 12 & 1.1e-15 \\
        256 & 12 & 12 & 2.6e-15 \\
        512 & 12 & 12 & 6.0e-15 \\
        \hline
    \end{tabular}
    \caption{Experiment 1: Algorithm~\ref{alg:ch_2_sys_general} with the minimum and with the fast marching method, random meshes.}
    \label{tab:exp1_fmm}
\end{table}

\begin{figure}[ht]
    \centering
    \includegraphics[width=\textwidth]{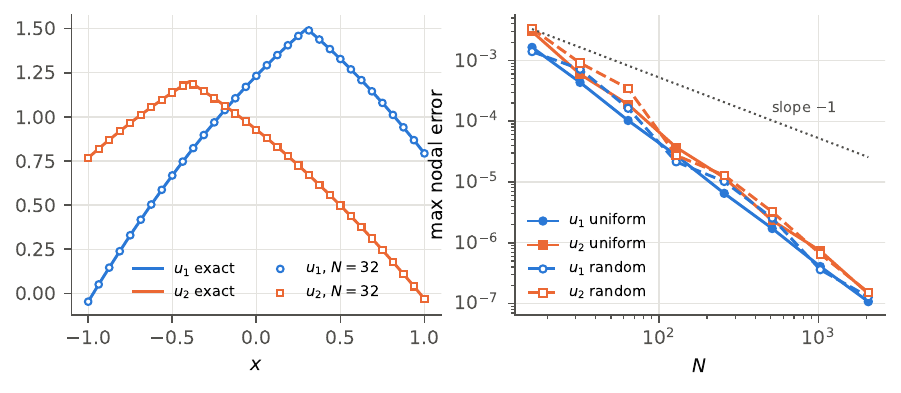}
    \caption{Experiment 1. Left: the manufactured solution and the discrete solution for $N = 32$. Right: maximum nodal errors.}
    \label{fig:exp1}
\end{figure}

\subsection{A star graph with three branches}\label{sec:exp2}

We take $M = 3$ branches of lengths $(1.0, 1.5, 0.7)$, coupling constants $(C_{1, m})_m = (0.6, 1.0, 1.5)$ and $(C_{2, m})_m = (0.9, 0.5, 1.2)$, and on each branch the manufactured solution
\begin{equation}\label{eq:exp2_u}
    \widetilde{u}_{i, m}(s) = U_i - \alpha_{i, m} s - \gamma_{i, m} s^2, \qquad s \in [0, \ell_m],
\end{equation}
where $U = (3.0, 2.7)$, $(\alpha_{1, m})_m = (0.8, 1.0, 1.2)$, $(\alpha_{2, m})_m = (0.9, 0.7, 1.1)$, $(\gamma_{1, m})_m = (0.2, 0.1, 0.3)$, and $(\gamma_{2, m})_m = (0.1, 0.3, 0.2)$. These functions are continuous at the junction, with the common values $U_1$ and $U_2$, and strictly decreasing from the junction toward the outer vertices. We set $f_{i, m} = |\widetilde{u}_{i, m}'| + C_{i, m} (\widetilde{u}_{i, m} - \widetilde{u}_{3-i, m})$, which is positive on every branch, and the exact boundary values $g_{i, m}(\ell_m) = \widetilde{u}_{i, m}(\ell_m)$, which are positive. Since $\widetilde{u}_{i, m}' < 0$, each $\widetilde{u}_{i, m}$ solves the linear equation $-\widetilde{u}_{i, m}' + C_{i, m} (\widetilde{u}_{i, m} - \widetilde{u}_{3-i, m}) = f_{i, m}$, which is what the iteration from the boundary vertex approximates.

Each branch $J_m$ has $n N_m^0$ uniform cells, with $N^0 = (8, 16, 8)$ and exact cell averages. For each $n$ we compute $\mathfrak{D}_+^{(m)}$, $\mathfrak{a}_+^{(m)}$, $T^{(m)}$, $\bm \xi^{(m)}$, and the coefficients $\mathfrak{q}$, $\mathfrak{r}$, $\mathfrak{s}$, as in Algorithm~\ref{alg:network_system_rev}. We fix the exact boundary pair $\mathfrak{g}_+^{(1)} = (2.0, 1.7)^T$ and determine the other boundary pairs by \eqref{eq:rev_gm}. Table~\ref{tab:exp2} reports the error $\max_m \left| T^{(m)} \mathfrak{g}_+^{(1)} + \bm \xi^{(m)} - \mathfrak{g}_{+, \mathrm{exact}}^{(m)} \right|$ of the boundary values on the other branches, the error $\left| \mathfrak{a}_+^{(1)} + \mathfrak{D}_+^{(1)} \mathfrak{g}_+^{(1)} - U \right|$ of the common junction value, and the maximum nodal error on all branches of the discrete solutions computed from the boundary values \eqref{eq:rev_gm}, all measured in the maximum norm over the relevant components and nodes. The junction jump is the largest componentwise difference between the junction values of the branches. Writing the inequalities that define $G$ as $\bm a_k \cdot \mathfrak{g} \leq b_k$, the signed distance of $\mathfrak{g}$ to $\partial G$ is
\[
    d(\mathfrak{g}, \partial G) = \min_k \frac{b_k - \bm a_k \cdot \mathfrak{g}}{\| \bm a_k \|_2},
    \qquad G = \{ \mathfrak{g} : \bm a_k \cdot \mathfrak{g} \leq b_k \text{ for all } k \},
\]
the distance to the nearest bounding line, positive inside $G$ and negative outside.

All three errors decrease at second order in $h$, while the junction jump remains at the level of rounding, as required by Lemma~\ref{lem:rev_continuity}. The exact boundary pair stays in the interior of $G$, about $0.44$ from its boundary, for every $n$. The interval for $g_{1, 1}(\ell_1) - g_{2, 1}(\ell_1)$ in \eqref{eq:rev_G_strip} stabilizes under refinement and contains the exact difference $0.3$ (Table~\ref{tab:exp2_interval}). Figures~\ref{fig:exp2_G} and \ref{fig:exp2_branches} show the admissible set and the branch profiles.

At $n = 2$, the interval for $g_{1, 1}(\ell_1) - g_{2, 1}(\ell_1)$ is $[-0.5202, 0.9322]$. The pair $(0.9174, 0.7114) \in G$ yields monotone profiles on every branch, whereas $(1.4984, 0.1305) \notin G$ violates monotonicity on branches $2$ and $3$, while the profiles on branch $1$ are monotone; the most violated inequality is the one with $j = 32$ in the first family of branch $2$, with violation $0.870$. Both choices give junction jumps of $1.3 \cdot 10^{-15}$, since the other boundary values are defined by \eqref{eq:rev_gm}. Thus continuity alone does not ensure admissibility, as expressed by Theorem~\ref{thm:rev_admissible}.

The unbounded direction of $G$ also agrees with the theory: $G + t (1, 1)^T \subset G$ for $t \geq 0$. Indeed, adding the same constant to both components does not change $|u_i'|$ or $u_1 - u_2$, so the system is invariant under this shift. In the discrete setting, $T^{(m)} (1, 1)^T = (1, 1)^T$ because the rows of $\mathfrak{D}_+^{(m)}$ sum to one, the conditions in Lemma~\ref{lem:rev_strip} depend only on the component difference $g_1 - g_2$, and the positivity conditions are preserved. Consequently, adding $t (1, 1)^T$, $t \geq 0$, preserves both monotonicity and nonnegative boundary values, and $G$, when it is not empty, is unbounded. For the pair inside $G$ above and $t \in \{1, 10, 100\}$, the smallest distance of $\mathfrak{g} + t (1, 1)^T$ to $\partial G$ is $0.509$.

\begin{table}[ht]
    \centering
    \resizebox{\textwidth}{!}{%
    \begin{tabular}{rrrrrrrrrr}
        \hline
        $n$ & cells & error $\mathfrak{g}^{(m)}$ & rate & error at junction & rate & error on branches & rate & jump & $d(\mathfrak{g}_+^{(1)}, \partial G)$ \\
        \hline
        1 & 32 & 7.27e-03 & -- & 3.08e-04 & -- & 7.27e-03 & -- & 1e-15 & +4.58e-01 \\
        2 & 64 & 1.82e-03 & 2.00 & 7.70e-05 & 2.00 & 1.82e-03 & 2.00 & 1e-15 & +4.47e-01 \\
        4 & 128 & 4.54e-04 & 2.00 & 1.92e-05 & 2.00 & 4.54e-04 & 2.00 & 4e-15 & +4.41e-01 \\
        8 & 256 & 1.13e-04 & 2.00 & 4.81e-06 & 2.00 & 1.13e-04 & 2.00 & 4e-15 & +4.38e-01 \\
        16 & 512 & 2.83e-05 & 2.00 & 1.20e-06 & 2.00 & 2.83e-05 & 2.00 & 8e-15 & +4.36e-01 \\
        32 & 1024 & 7.09e-06 & 2.00 & 3.01e-07 & 2.00 & 7.09e-06 & 2.00 & 1e-14 & +4.35e-01 \\
        64 & 2048 & 1.77e-06 & 2.00 & 7.52e-08 & 2.00 & 1.77e-06 & 2.00 & 2e-14 & +4.35e-01 \\
        \hline
    \end{tabular}}
    \caption{Experiment 2: errors of the boundary values, of the junction values, and of the solutions on the branches, the jump at the junction, and the distance of the exact boundary values to $\partial G$.}
    \label{tab:exp2}
\end{table}

\begin{table}[ht]
    \centering
    \begin{tabular}{rrrr}
        \hline
        $n$ & lower end & upper end & exact $g_{1,1} - g_{2,1}$ \\
        \hline
        1 & -0.5629 & 0.9478 & 0.3000 \\
        2 & -0.5202 & 0.9322 & 0.3000 \\
        4 & -0.4993 & 0.9235 & 0.3000 \\
        8 & -0.4889 & 0.9188 & 0.3000 \\
        16 & -0.4838 & 0.9165 & 0.3000 \\
        32 & -0.4812 & 0.9153 & 0.3000 \\
        64 & -0.4799 & 0.9147 & 0.3000 \\
        \hline
    \end{tabular}
    \caption{Experiment 2: the interval for $g_{1, 1}(\ell_1) - g_{2, 1}(\ell_1)$ defined by \eqref{eq:rev_G_strip}.}
    \label{tab:exp2_interval}
\end{table}

\begin{figure}[ht]
    \centering
    \includegraphics[width=0.55\textwidth]{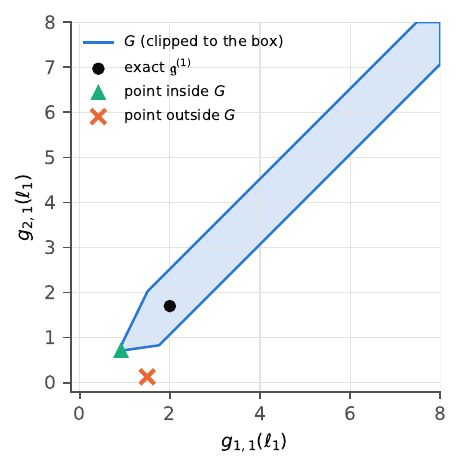}
    \caption{Experiment 2, $n = 2$: the admissible set $G$ (clipped to the plotted box; it continues in the direction $(1, 1)^T$), the exact boundary values on $J_1$, and the two test points.}
    \label{fig:exp2_G}
\end{figure}

\begin{figure}[ht]
    \centering
    \includegraphics[width=\textwidth]{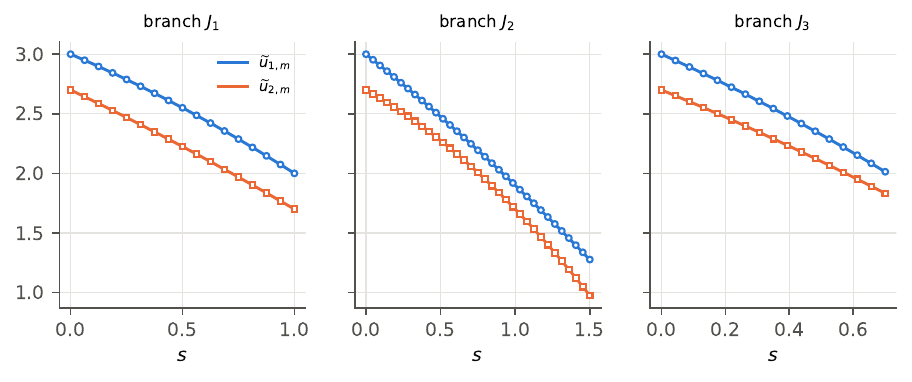}
    \caption{Experiment 2, $n = 2$: the manufactured solution (lines) and the discrete solution computed from the exact $\mathfrak{g}_+^{(1)}$ and \eqref{eq:rev_gm} (markers).}
    \label{fig:exp2_branches}
\end{figure}

\subsection{A star graph with a singular branch}\label{sec:exp3}

Let $t^* \approx 1.256431$ be the positive root of $\mathfrak{B}(t) = 1/2$. We keep branches $J_1$ and $J_3$ of \S\ref{sec:exp2} with $n = 1$, and replace $J_2$ by a branch with $C_{1, 2} = C_{2, 2} = 1$, three uniform cells of length $t^*$, and
\[
    f_{1, 2}(s) = 1 + 0.5 \sin s, \qquad f_{2, 2}(s) = 0.8 + 0.3 \cos 2s.
\]
By Theorem~\ref{thrm:det_D_product}, $\mathfrak{D}_+^{(2)}$ is singular; the computed determinant is $\det(\mathfrak{D}_+^{(2)}) = 2.2 \cdot 10^{-16}$. By \eqref{eq:diff_junction} in Lemma~\ref{lem:diff_row}, on a singular branch the difference at the junction is the same for all boundary values:
\[
    \widetilde{w}_{1, 2}(0) - \widetilde{w}_{2, 2}(0)
    = \mathfrak{a}_{1, +}^{(2)} - \mathfrak{a}_{2, +}^{(2)} \approx 0.202339.
\]
Table~\ref{tab:exp3_pinned} confirms this identity for random boundary values; the junction values are computed by the iteration from the boundary vertex.

\begin{table}[ht]
    \centering
    \begin{tabular}{rrrrr}
        \hline
        $g_{1,2}$ & $g_{2,2}$ & $\widetilde w_{1,2}(0)$ & $\widetilde w_{2,2}(0)$ & difference \\
        \hline
        1.5355 & 2.8514 & 6.2106 & 6.0083 & 0.2023 \\
        0.4325 & 2.8459 & 5.6564 & 5.4541 & 0.2023 \\
        0.9355 & 1.2700 & 5.1199 & 4.9176 & 0.2023 \\
        2.4831 & 1.2276 & 5.8725 & 5.6702 & 0.2023 \\
        1.6488 & 0.0827 & 4.8829 & 4.6806 & 0.2023 \\
        \hline
    \end{tabular}
    \caption{Experiment 3: the junction values on the singular branch for random boundary values.}
    \label{tab:exp3_pinned}
\end{table}

Since $\det(\mathfrak{D}_+^{(1)}) \approx 0.222658$, the compatibility condition \eqref{eq:rank_one_compat} requires $g_{1, 1}(\ell_1) - g_{2, 1}(\ell_1) = \delta^* \approx -0.140920$. We set $g_{2, 1}(\ell_1) = 1$ and $g_{1, 1}(\ell_1) = 1 + \delta^* + \varepsilon$. On $J_2$, we choose $g_{1, 2}(\ell_2)$ freely and determine $g_{2, 2}(\ell_2)$ from the first row of \eqref{eq:ch_3_mathfrakD_telescope}, which is the formula for the rank-one case; on $J_3$, we use \eqref{eq:rev_gm}. Table~\ref{tab:exp3_condition} lists the jumps of the junction values over the three branches.

For $\varepsilon = 0$, the solutions are continuous at the junction for every choice of $g_{1, 2}(\ell_2)$: in the singular case there is a one-parameter family of continuous solutions for fixed data on $J_1$. For $\varepsilon \neq 0$, the first component remains continuous and the second has a jump of $|\det(\mathfrak{D}_+^{(1)})\, \varepsilon|$, as predicted by \eqref{eq:diff_telescope} and \eqref{eq:rank_one_compat}. Table~\ref{tab:exp3_condition} confirms both cases. The resulting values of $g_{2, 2}(\ell_2)$ are negative, so these boundary values are not feasible; this test verifies only the continuity conditions, without asserting admissibility.

\begin{table}[ht]
    \centering
    \begin{tabular}{rrrrrr}
        \hline
        $\varepsilon$ & $g_{1,2}$ & $g_{2,2}$ from the formula & jump in $u_1$ & jump in $u_2$ & $\det(\mathfrak{D}_+^{(1)})\,\varepsilon$ \\
        \hline
        +0.0 & 0.5 & -4.5416 & 3.3e-15 & 1.55e-15 & +0.00e+00 \\
        +0.0 & 1.5 & -5.5416 & 4.0e-15 & 1.55e-15 & +0.00e+00 \\
        +0.1 & 0.5 & -4.4038 & 3.1e-15 & 2.23e-02 & +2.23e-02 \\
        +0.1 & 1.5 & -5.4038 & 3.6e-15 & 2.23e-02 & +2.23e-02 \\
        -0.2 & 0.5 & -4.8172 & 4.0e-15 & 4.45e-02 & -4.45e-02 \\
        -0.2 & 1.5 & -5.8172 & 4.0e-15 & 4.45e-02 & -4.45e-02 \\
        \hline
    \end{tabular}
    \caption{Experiment 3: jumps of the junction values over the three branches for compatible data ($\varepsilon = 0$) and perturbed data. The boundary values on $J_1$ are $g_{2, 1}(\ell_1) = 1$ and $g_{1, 1}(\ell_1) = 1 + \delta^* + \varepsilon$, so $\varepsilon$ measures the violation of \eqref{eq:rank_one_compat}. The predicted jump in the second component is $|\det(\mathfrak{D}_+^{(1)})\, \varepsilon|$.}
    \label{tab:exp3_condition}
\end{table}

Figure~\ref{fig:exp3_det} illustrates the determinant formula \eqref{eq:det_D_product}: it shows $\det(\mathfrak{D}_+^{(m)})$ for equal coupling constants $C_{1, m} = C_{2, m}$ on uniform meshes with $N_m = 1, 2, 3$ cells, together with the cell factor $\mathfrak{B}(C_1 h) + \mathfrak{B}(C_2 h) - 1$ for three ratios $C_2 / C_1$. At $C h = t^*$, the determinant changes sign for odd $N_m$ and touches zero without changing sign for even $N_m$, in agreement with Theorem~\ref{thrm:det_D_product}. At nine sample points, the product formula agrees with $\mathfrak{b}_{1, +}^{(m)} + \mathfrak{b}_{2, +}^{(m)} - 1$ computed from the definitions to within $4.4 \cdot 10^{-16}$.

\begin{figure}[ht]
    \centering
    \includegraphics[width=\textwidth]{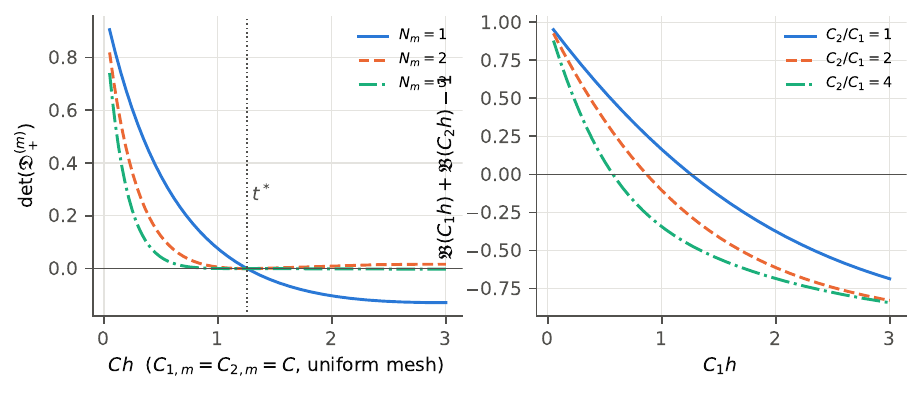}
    \caption{Experiment 3. Left: $\det(\mathfrak{D}_+^{(m)})$ for $C_{1, m} = C_{2, m} = C$ on a uniform mesh with cell size $h$. Right: the factor of one cell in \eqref{eq:det_D_product}.}
    \label{fig:exp3_det}
\end{figure}

\clearpage
\bibliographystyle{plain}
\bibliography{bib_1d_paper}

@article{CrandallLions1984,
    author = "Michael G. Crandall and Pierre-Louis Lions",
    title = "Two Approximations of Solutions of {H}amilton-{J}acobi Equations",
    journal = "Mathematics of Computation",
    volume = "43",
    number = "167",
    month = jul,
    year = "1984",
    pages = "1--19"
}

@article{ImbertMonneau,
    author = "Cyril Imbert and R{\'e}gis Monneau",
    title = "Flux-Limited Solutions for Quasi-Convex {H}amilton-{J}acobi Equations on Networks",
    journal = "Annales Scientifiques de l'{\'E}cole Normale Sup{\'e}rieure (4)",
    volume = "50",
    number = "2",
    year = "2017",
    pages = "357--448"
}

@article{Sethian,
    author = "J. A. Sethian",
    title = "Fast Marching Methods",
    journal = "SIAM Review",
    volume = "41",
    number = "2",
    year = "1999",
    pages = "199--235"
}

@article{SchiebornCamilli,
    author = "Dirk Schieborn and Fabio Camilli",
    title = "Viscosity Solutions of {E}ikonal Equations on Topological Networks",
    journal = "Calculus of Variations and Partial Differential Equations",
    volume = "46",
    year = "2013",
    pages = "671--686",
    doi = "10.1007/s00526-012-0498-z"
}

@article{AchdouCamilliCutriTchou2013,
    author = "Yves Achdou and Fabio Camilli and Alessandra Cutr{\`\i} and Nicoletta Tchou",
    title = "{H}amilton-{J}acobi Equations Constrained on Networks",
    journal = "NoDEA Nonlinear Differential Equations and Applications",
    volume = "20",
    year = "2013",
    pages = "413--445"
}

@article{LionsSouganidis2016,
    author = "Pierre-Louis Lions and Panagiotis E. Souganidis",
    title = "Viscosity Solutions for Junctions: Well Posedness and Stability",
    journal = "Atti della Accademia Nazionale dei Lincei. Rendiconti Lincei. Matematica e Applicazioni",
    volume = "27",
    year = "2016",
    pages = "535--545"
}

@article{CamilliMarchiSchieborn2013,
    author = "Fabio Camilli and Claudio Marchi and Dirk Schieborn",
    title = "The Vanishing Viscosity Limit for {H}amilton-{J}acobi Equations on Networks",
    journal = "Journal of Differential Equations",
    volume = "254",
    year = "2013",
    pages = "4122--4143"
}

@article{SiconolfiSorrentino2018,
    author = "Antonio Siconolfi and Alfonso Sorrentino",
    title = "Global Results for Eikonal {H}amilton-{J}acobi Equations on Networks",
    journal = "Analysis \& PDE",
    volume = "11",
    year = "2018",
    pages = "171--211"
}

@article{CamilliFestaSchieborn2013,
    author = "Fabio Camilli and Adriano Festa and Dirk Schieborn",
    title = "An Approximation Scheme for a {H}amilton-{J}acobi Equation Defined on a Network",
    journal = "Applied Numerical Mathematics",
    volume = "73",
    year = "2013",
    pages = "33--47"
}

@article{CostesequeLebacqueMonneau2015,
    author = "Guillaume Costeseque and Jean-Patrick Lebacque and R{\'e}gis Monneau",
    title = "A Convergent Scheme for {H}amilton-{J}acobi Equations on a Junction: Application to Traffic",
    journal = "Numerische Mathematik",
    volume = "129",
    year = "2015",
    pages = "405--447"
}

@article{GuerandKoumaiha2019,
    author = "Jessica Guerand and Marwa Koumaiha",
    title = "Error Estimates for a Finite Difference Scheme Associated with {H}amilton-{J}acobi Equations on a Junction",
    journal = "Numerische Mathematik",
    year = "2019",
    note = "arXiv:1502.07158; volume and pages to be checked"
}

@article{GeeVladimirsky2022,
    author = "Marissa Gee and Alexander Vladimirsky",
    title = "Optimal Path-Planning with Random Breakdowns",
    journal = "IEEE Control Systems Letters",
    year = "2022",
    note = "arXiv:2109.06910; volume and pages to be checked"
}

\end{document}